\documentclass[psamsfonts,a4paper,11pt]{amsart}
\usepackage[all,arc]{xy}
\usepackage[all]{xy}
\usepackage{latexsym}
\usepackage{amssymb}
\usepackage{amsmath}
\newtheorem{definition}{Definition}[section]
\newtheorem{lemma}[definition]{Lemma}
\newtheorem{prop}[definition]{Proposition}
\newtheorem{theorem}[definition]{Theorem}
\newtheorem{cor}[definition]{Corollary}
\newtheorem{conj}[definition]{Conjecture}

\newtheorem{ques}[definition]{Question}
\newtheorem{remark}[definition]{Remark}
\theoremstyle{definition}
\newtheorem{fact}[definition]{Fact}

\newcommand*{\AR}{\mathop{{\rm AR}}\nolimits}

\newcommand*{\CB}{\mathop{{\rm CB}}\nolimits}
\newcommand*{\CM}{\mathop{{\rm CM}}\nolimits}

\newcommand*{\End}{\mathop{{\rm End}}\nolimits}

\newcommand*{\Ext}{\mathop{{\rm Ext}}\nolimits}
\newcommand*{\Hom}{\mathop{{\rm Hom}}\nolimits}

\newcommand*{\rad}{\mathop{{\rm rad}}\nolimits}

\newcommand{\Zg}{\mathop{{\rm Zg}}\nolimits}

\newcommand{\ZCM}{\mathop{{\rm ZCM}}\nolimits}

\newcommand*{\ex}{\exists}

\newcommand*{\ms}{\models}

\newcommand*{\ov}{\overline}
\newcommand*{\seq}{\subseteq}
\newcommand*{\sep}{\supseteq}
\newcommand*{\sm}{\setminus}

\newcommand*{\bsm}{\left(\begin{smallmatrix}}
\newcommand*{\esm}{\end{smallmatrix}\right)}
\newcommand*{\bp}{\begin{pmatrix}}
\newcommand*{\ep}{\end{pmatrix}}

\newcommand*{\fty}{\infty}

\newcommand*{\wg}{\wedge}

\newcommand*{\wt}{\widetilde}

\newcommand*{\xr}{\xrightarrow}

\newcommand*{\mD}{\mathcal{D}}
\newcommand*{\mm}{\mathbf{m}}

\newcommand*{\Z}{\mathbb{Z}}

\newcommand*{\be}{\beta}

\newcommand*{\lam}{\lambda}

\newcommand*{\om}{\omega}

\renewcommand*{\phi}{\varphi}

\begin{document}

\footskip=30pt

\date{}

\title[]{The Ziegler spectrum and Ringel's quilt of the $A$-infinity plane singularity}



\author[]{Gena Puninski}

\address[G.~Puninski]{Belarusian State University, Faculty of Mechanics and Mathematics, av. Nezalezhnosti 4,
Minsk 220030, Belarus}
\email{punins@mail.ru}




\begin{abstract}
We describe the Cohen--Macaulay part of the Ziegler spectrum and calculate Ringel's quilt of the category of
finitely generated Cohen--Macaulay modules over the $A$-infinity plane singularity.
\end{abstract}

\maketitle

\section{Introduction}\label{S-intro}

The original objective of this paper was to describe the Cohen--Macaulay part of the Ziegler spectrum over the
$A_{\fty}$ plane singularity $R$. With a decent knowledge of finitely generated points in this topological space
and Model Theory of Modules this proved to be a task of reasonable complexity. Surprisingly we found that all
non-finitely generated points in this part of the Ziegler spectrum are natural ones: the integral closure
$\wt R$ of $R$, its quotient ring $Q$ and a generic module $G$.

This classification was achieved by the so-called 'interval method': instead of classifying points of this space,
i.e. indecomposable pure injective modules, we calculate some intervals in the lattice of finitely generated
subfunctors of the functor $\Hom(R,-)$ from the category of finitely generated Cohen--Macaulay modules to abelian
groups.

If the lattice structure of such interval is known, then one can use Ziegler's result that each indecomposable pure
injective module  opening this interval is uniquely determined by a nontrivial filter defined by its realization.
This approach resembles the method (hence the name) used by Gelfand--Ponomarev \cite{G-P}, and later by
Ringel \cite{Rin75}, to classify indecomposable finite dimensional modules over certain classes of finite dimensional
algebras.

As for now we have to rely on classification of finitely generated Cohen--Macaulay modules to describe the infinite
part of the Ziegler spectrum. Furthermore it could be a nontrivial task to recover an indecomposable pure injective
module from the pp-type it realizes, but this is quite straightforward in the example of our interest.

We completely describe the topology of the Cohen--Macaulay part of the Ziegler spectrum of $R$, in particular
prove that the Cantor--Bendixson rank of this space equals 2; and the same is the value of the $m$-dimension
of the lattice of pp-formulae of the theory of Cohen--Macaulay modules. From this point of view the category
of finitely generated Cohen--Macaulay modules over $R$ resembles the category of finite dimensional modules over
tame hereditary finite dimensional algebras.

However some peculiarities were also observed. For instance one finitely generated Cohen--Macaulay module has
Cantor--Bendixson rank $1$, in particular is non-isolated. This module is at the end of an almost split sequence
in the category of all (finitely generated or not) Cohen--Macaulay modules whose source is an infinitely generated
pure injective module. Furthermore, contrary to the case of finite dimensional algebras, none of indecomposable
infinitely generated pure injective Cohen--Macaulay modules is a direct summand of a direct product of finitely
generated ones.

Some non-finitely generated points of the Ziegler spectrum are direct limits of finitely generated Cohen--Macaulay
modules (along a ray of irreducible morphism), so we use these points to glue the two components of the
Auslander--Reiten quiver of $R$ obtaining a 2-dimensional surface  we call the Ringel quilt of $R$. The additional
devices used in this knitting procedure are irreducible morphisms and almost split sequences with infinitely
generated terms.

This gluing procedure was implemented by Ringel \cite{Rin00} when investigating modules over domestic string
algebras and gives a new insight on the global geometric structure of the category of finite dimensional modules.

We will show that over the $A_{\fty}$ plane singularity the category of finitely generated Cohen--Macaulay
modules can be neatly lodged on the M\"obius stripe. For instance morphisms in this category amount to easily
controlled walks on this surface. Using this representation we prove that the nilpotency index of the radical
of this category equals $\om+ 2$.

We will show that the problem of classifying points of the Ziegler spectrum of higher dimensional $A_{\fty}$
singularities $R_d$, $d\geq 2$ is wild. However there is a good chance to calculate the closure, in the Ziegler
spectrum, of the set of finitely generated Cohen--Macaulay points; but this is rather a task for future.

Many statements in this paper admit obvious generalizations, bur we decided to suppress writing them down.
A general philosophy is that when one starts developing a new theory a carefully calculated example creates a
better guide than general statements.

The projected audience of this paper is twofold: the experts in commutative algebra who are interested in learning
methods of model theory of modules; and also people from model theory of modules who are keen to step on the well
fertilized turf with a new exploring tool. It is quite difficult to make this text easily accessible to both
groups, and our exposition is clearly biased: we will be quite meticulous in explaining basics of the theory of
Cohen--Macaulay modules, but more sketchy on items in model theory of modules. For those the reader is referred
to Mike Prest's book \cite{Preb} which is not so short, but contains almost all references we need.

The author is indebted to Ivo Herzog for useful comments on preliminary versions of the paper.

\section{Basics}

Let $F$ be an algebraically closed field of characteristic not equal to $2$. In fact - see \cite[Rem. 1.2.22]{B-G} -
the results are likely to be true also when characteristic is equal to 2, but it is difficult to find proper
references. Let $R= F[[x,y]]/(x^2)$ be the factor of the power series ring by the ideal generated by $x^2$, the
so-called \emph{$A_{\fty}$ plane singularity} - see \cite{BGS}. It is known from this paper
(see also \cite[Thm. 14.16]{L-W}) that $R$ has a countable Cohen--Macaulay representation type, i.e. only countably
many indecomposable finitely generated Cohen--Macaulay modules.

We need few easily verified facts on the structure of $R$. Clearly $R$ is a commutative complete local noetherian
ring whose maximal ideal $\mm$ is generated by $x$ and $y$. From $x^2= 0$ and $R/xR\cong F[[y]]$ we conclude that
$xR\subset \mm$ are the only prime ideals of $R$, in particular $R$ has Krull dimension 1.

Recall that a (commutative noetherian local) ring $S$ is said to be \emph{Gorenstein} if the regular module $S_S$ has
a finite injective dimension; for instance $R$ is Gorenstein. Namely note that $y$ is a non-zero divisor in $R$
and $R/yR\cong F[[x]]/(x^2)$. Because the socle of this ring is simple, \cite[Prop. 21.5]{Eis} yields that
$R/yR$ is $0$-dimensional Gorenstein. Since $R$ is local, we conclude by \cite[3.1.19 b)]{B-H}.

Furthermore $R$ is clearly uniform, so is its quotient ring $Q$. It is easily checked that it suffices to invert $y$
(or any nonzero divisor) to get $Q$. Namely $s\in R$ is a non-zero divisor iff $s\notin xR$, hence $s= f(y)+ xg(y)$,
where $f\neq 0$. Multiplying by a unit in $F[[y]]$ we may assume that $s= y^n+ xg$, hence $s^{-1}= (y^n- xg) y^{-2n}$.

Since $R$ is Gorenstein, it follows that $Q$ is an indecomposable injective module. Furthermore $Q/R$ is the
injective envelope of the simple module $F$, hence the minimal injective cogenerator in the category of
$R$-modules. In particular $0\to R\to Q\to Q/R\to 0$ is the minimal injective resolution of $R$, therefore
$R$ has injective dimension 1.

Recall that $\wt R$ denotes the integral closure of $R$ in $Q$. The following is also straightforward.

\begin{remark}\label{val}
$\wt R$ is a non-noetherian valuation ring of Krull dimension 1 with the following chain of principal ideals:

$$
\wt R\supset y\wt R\supset y^2\wt R\supset \ldots \supset xy^{-1}\wt R\supset x\wt R\supset xy\wt R\supset \dots\,.
$$

In particular the ideal $\cap_n y^n \wt R= \cup_{m\in \Z}\, xy^m \Z$ is nilpotent and non-principal. Also
$\wt R$ is not finitely generated as an $R$-module.
\end{remark}
\begin{proof}
It is easily seen that $\wt R= R+ xQ$, where the inclusion $\sep$ is obvious, since every element of $xQ$ is
nilpotent.

As above (up to a multiplicative unit) every element from $\wt R\sm xQ$ is of the form $r= y^n+ x g(y)$.
Furthermore $r= y^n (1+ xy^{-n} g)$, where the last factor is a unit in $\wt R$. Thus the principal ideal
generated by $r$ equals $y^n \wt R$.

Similarly every cyclic $\wt R$-submodule of $xQ$ is generated by $xy^m$ for some integer $m$.
\end{proof}

Recall, see \cite[p. 1]{Yosh}, a definition of maximal Cohen--Macaulay modules. Let $S$ be a commutative noetherian
local ring of Krull dimension $d$ whose residue field is $F$. An $S$-module $N$ is said to be \emph{maximal
Cohen--Macaulay}, $\CM$ for short, if $\Ext^i(F,N)= 0$ for all $0\leq i< d$. We take this as a definition even
when $N$ is not finitely generated.

In our case this boils down to the following.

\begin{remark}\label{cm-def}
An $R$-module $N$ is Cohen-Macaulay if and only if $N$ has no $y$-torsion: $ny=0$ for some $n\in N$ implies $n=0$.
\end{remark}
\begin{proof}
Because $R$ is 1-dimensional, $N$ is $\CM$ iff $\Hom(F,N)=0$, i.e. $nx= ny= 0$ for some $n\in N$ implies $n= 0$.
Since $x^2=0$, this is clearly equivalent to the absence of $y$-torsion.
\end{proof}

\section{Finitely generated $\CM$-modules}\label{S-fg}

In this section we will recall the classification of indecomposable finitely generated $\CM$-modules over $R$.
By the above description we are in the framework of Bass' ubiquity paper \cite{Bass}, see also comments in
\cite[Thm. 4.18]{L-W}. For instance every indecomposable finitely generated $\CM$-module over $R$ is isomorphic
to an ideal of $R$ (equivalently to a finitely generated module between $R$ and $\wt R$). One can be even more
precise.

Let $I_n= (x, y^n)$, $n\geq 0$ be the ideal of $R$ generated by $x$ and $y^n$, in particular $I_0= R$ and
$I_1= \mm$. Further we set $I_{\fty}= xR$. The following is well known - see \cite[Exam. 6.5]{Yosh} or \cite{Sch}.

\begin{fact}\label{class-fg}
Each indecomposable finitely generated $\CM$-module over $R$ is isomorphic to $I_n$ for some $n\geq 0$, or to
$I_{\fty}$.
\end{fact}

For future use we represent these modules by the following diagrams.

\vspace{5mm}

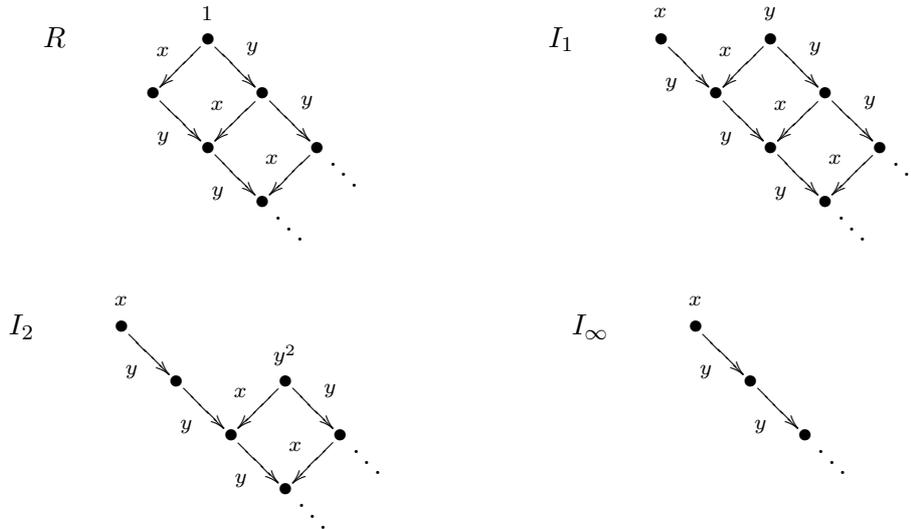
\begin{center}
$
R\hspace{10mm}
\xymatrix@C=12pt@R=12pt{%
&*+={\bullet}\ar[dr]^y\ar[dl]_x\ar@{}+<0pt,10pt>*{_1}&&&\\
*+={\bullet}\ar[dr]_y&&*+={\bullet}\ar[dr]^y\ar[dl]_x&&&\\
&*+={\bullet}\ar[dr]_y&&*+={\bullet}\ar[dl]_x\ar@{}+<6pt,-6pt>*{.}\ar@{}+<10pt,-10pt>*{.}\ar@{}+<14pt,-14pt>*{.}\\
&&*+={\bullet}\ar@{}+<6pt,-6pt>*{.}\ar@{}+<10pt,-10pt>*{.}\ar@{}+<14pt,-14pt>*{.}&
}
$
\hspace{12mm}
$
I_1\hspace{10mm}
\xymatrix@C=12pt@R=12pt{%
*+={\bullet}\ar[dr]_y\ar@{}+<0pt,10pt>*{_x}&&*+={\bullet}\ar[dr]^y\ar[dl]_x\ar@{}+<0pt,10pt>*{_y}&\\
&*+={\bullet}\ar[dr]_y&&*+={\bullet}\ar[dl]_x\ar[dr]^y\\
&&*+={\bullet}\ar[dr]_y&&*+={\bullet}\ar[dl]_x\ar@{}+<6pt,-6pt>*{.}\ar@{}+<10pt,-10pt>*{.}\ar@{}+<14pt,-14pt>*{.}&\\
&&&*+={\bullet}\ar@{}+<6pt,-6pt>*{.}\ar@{}+<10pt,-10pt>*{.}\ar@{}+<14pt,-14pt>*{.}&&
}
$
\end{center}

\vspace{5mm}

\begin{center}
$
I_2\hspace{10mm}
\xymatrix@C=12pt@R=12pt{%
*+={\bullet}\ar@{}+<0pt,10pt>*{_x}\ar[dr]_y&&&\\
&*+={\bullet}\ar[dr]_y&&*+={\bullet}\ar@{}+<0pt,10pt>*{_{y^2}}\ar[dr]^y\ar[dl]_x&&&\\
&&*+={\bullet}\ar[dr]_y&&*+={\bullet}\ar[dl]_x\ar@{}+<6pt,-6pt>*{.}\ar@{}+<10pt,-10pt>*{.}\ar@{}+<14pt,-14pt>*{.}\\
&&&*+={\bullet}\ar@{}+<6pt,-6pt>*{.}\ar@{}+<10pt,-10pt>*{.}\ar@{}+<14pt,-14pt>*{.}&
}
$
\hspace{12mm}
$
I_{\infty}\hspace{10mm}
\xymatrix@C=12pt@R=12pt{%
*+={\bullet}\ar@{}+<0pt,10pt>*{_x}\ar[dr]_y&&&\\
&*+={\bullet}\ar[dr]_y&&&&\\
&&*+={\bullet}\ar@{}+<6pt,-6pt>*{.}\ar@{}+<10pt,-10pt>*{.}\ar@{}+<14pt,-14pt>*{.}&
}
$
\end{center}

\vspace{5mm}

Furthermore the following is the Auslander--Reiten quiver of the category of finitely generated $\CM$-modules.

$$
\vcenter{%
\def\labelstyle{\displaystyle}
\xymatrix@C=20pt@R=10pt{%
R\ar@<.5ex>[r]&I_1\ar@<.5ex>[r]\ar@<.5ex>[l]&I_2\ar@<.5ex>[r]\ar@<.5ex>[l]&\ar@<.5ex>[l]&\dots&
*+={I_{\infty}}\ar@(ur,dr)\\
}}
$$

\vspace{2mm}

Here the (left to right) irreducible morphisms $I_n\to I_{n+1}$ are given by multiplication by $y$; and the
(right to left) irreducible morphisms $I_{n+1}\to I_n$ are inclusions. Also the arrow $I_{\fty}\to I_{\fty}$
is given by multiplication by $y$.

For instance each $I_n$, $1\leq n< \fty$ is the source and the sink of the following $\AR$-sequence in the
category of finitely generated $\CM$-modules.

$$
0\to I_n\xr{\bsm 1\\y\esm} I_{n-1}\oplus I_{n+1}\xr{(y,-1)} I_n\to 0\,.
$$

Here, because we consider right modules, their morphisms act on the left.

Being Gorenstein, $R$ is a projective and injective object in the category of finitely generated $\CM$-modules,
therefore no $\AR$-sequence starts or ends in $R$.

These $\AR$-sequences and irreducible morphisms are clearly visible on the above diagrams. For instance the
irreducible map $I_1\xr{y} I_2$ amounts to dividing $x\in I_1$ by $y$.

\vspace{5mm}

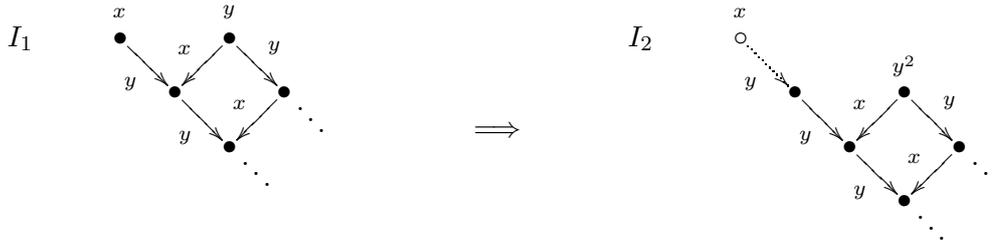
\begin{center}
$
I_1\hspace{10mm}
\xymatrix@C=12pt@R=12pt{%
*+={\bullet}\ar@{}+<0pt,10pt>*{_x}\ar[dr]_y&&*+={\bullet}\ar[dr]^y\ar[dl]_{x}\ar@{}+<0pt,10pt>*{_y}&\\
&*+={\bullet}\ar[dr]_y&&*+={\bullet}\ar[dl]_x\ar@{}+<6pt,-6pt>*{.}\ar@{}+<10pt,-10pt>*{.}\ar@{}+<14pt,-14pt>*{.}&&\\
&&*+={\bullet}\ar@{}+<6pt,-6pt>*{.}\ar@{}+<10pt,-10pt>*{.}\ar@{}+<14pt,-14pt>*{.}&
}
$
\hspace{5mm}
$\xymatrix{\\\Longrightarrow}$
\hspace{10mm}
$
I_2\hspace{10mm}
\xymatrix@C=12pt@R=12pt{%
*+={\circ}\ar@{}+<0pt,10pt>*{_x}\ar@{.>}[dr]_y&&&\\
&*+={\bullet}\ar[dr]_y&&*+={\bullet}\ar@{}+<0pt,10pt>*{_{y^2}}\ar[dr]^y\ar[dl]_x&&&\\
&&*+={\bullet}\ar[dr]_y&&*+={\bullet}\ar[dl]_x\ar@{}+<6pt,-6pt>*{.}\ar@{}+<10pt,-10pt>*{.}\ar@{}+<14pt,-14pt>*{.}\\
&&&*+={\bullet}\ar@{}+<6pt,-6pt>*{.}\ar@{}+<10pt,-10pt>*{.}\ar@{}+<14pt,-14pt>*{.}&
}
$
\end{center}

\vspace{5mm}

Similarly the inclusion $I_1\subset R$ is shown by completing the square on the following diagram.

$$
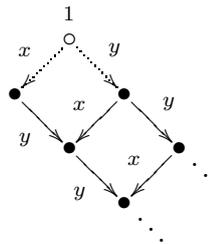
\xymatrix@C=12pt@R=12pt{%
&*+={\circ}\ar@{.>}[dl]_x\ar@{.>}[dr]^y\ar@{}+<0pt,10pt>*{_1}&&\\
*+={\bullet}\ar[dr]_y&&*+={\bullet}\ar[dr]^y\ar[dl]_{x}&\\
&*+={\bullet}\ar[dr]_y&&*+={\bullet}\ar[dl]_x\ar@{}+<6pt,-6pt>*{.}\ar@{}+<10pt,-10pt>*{.}\ar@{}+<14pt,-14pt>*{.}&&\\
&&*+={\bullet}\ar@{}+<6pt,-6pt>*{.}\ar@{}+<10pt,-10pt>*{.}\ar@{}+<14pt,-14pt>*{.}&
}
$$

\vspace{3mm}

It follows from the above description that the $\AR$-quiver of $R$ has two components, - the first contains all
'finite' modules $I_n$; and the second consists of $I_{\fty}$. We show, on Figure \ref{fig1}, the extended version
of this quiver which contains both components, but also includes commutativity relations between irreducible maps.

\vspace{3mm}

\begin{figure}[b]
$$
\xymatrix@C=16pt@R=16pt{%
&&&*+={\bullet}\ar@{}+<8pt,8pt>*{_{I_{\fty}}}\ar[rd]^y
\ar@{}+<-6pt,6pt>*{.}\ar@{}+<-10pt,10pt>*{.}\ar@{}+<-14pt,14pt>*{.}&\\
*+={\bullet}\ar@{}+<-10pt,0pt>*{_R}\ar[rd]^y\ar@{}+<6pt,6pt>*{.}\ar@{}+<10pt,10pt>*{.}\ar@{}+<14pt,14pt>*{.}
\ar@{}+<0pt,20pt>*{.}\ar@{}+<0pt,26pt>*{.}\ar@{}+<0pt,32pt>*{.}&&
*+={\bullet}\ar@{}+<-10pt,0pt>*{_{I_2}}\ar[ld]\ar[rd]^y\ar@{}+<6pt,6pt>*{.}\ar@{}+<10pt,10pt>*{.}\ar@{}+<14pt,14pt>*{.}
\ar@{}+<-6pt,6pt>*{.}\ar@{}+<-10pt,10pt>*{.}\ar@{}+<-14pt,14pt>*{.}&&
*+={\bullet}\ar@{}+<8pt,8pt>*{_{I_{\fty}}}\ar[rd]^y\\
&*+={\bullet}\ar@{}+<-10pt,0pt>*{_{I_1}}\ar[ld]\ar[rd]^y&&*+={\bullet}\ar@{}+<-10pt,0pt>*{_{I_3}}\ar[ld]\ar[rd]^y
\ar@{}+<6pt,6pt>*{.}\ar@{}+<10pt,10pt>*{.}\ar@{}+<14pt,14pt>*{.}
&&*+={\bullet}\ar@{}+<8pt,8pt>*{_{I_{\fty}}}\ar[ld]
\ar@{}+<6pt,-6pt>*{.}\ar@{}+<10pt,-10pt>*{.}\ar@{}+<14pt,-14pt>*{.}\\
*+={\bullet}\ar@{}+<-10pt,0pt>*{_R}\ar[rd]^y&&*+={\bullet}\ar@{}+<-10pt,0pt>*{_{I_2}}\ar[ld]\ar[rd]^y
&&*+={\bullet}\ar@{}+<-10pt,0pt>*{_{I_4}}\ar[ld]
\ar@{}+<6pt,-6pt>*{.}\ar@{}+<10pt,-10pt>*{.}\ar@{}+<14pt,-14pt>*{.}\\
&*+={\bullet}\ar@{}+<-10pt,0pt>*{_{I_1}}\ar[ld]\ar[rd]^y&&*+={\bullet}\ar@{}+<-10pt,0pt>*{_{I_3}}\ar[ld]
\ar@{}+<6pt,-6pt>*{.}\ar@{}+<10pt,-10pt>*{.}\ar@{}+<14pt,-14pt>*{.}&\\
*+={\bullet}\ar@{}+<-10pt,0pt>*{_R}\ar@{}+<6pt,-6pt>*{.}\ar@{}+<10pt,-10pt>*{.}\ar@{}+<14pt,-14pt>*{.}
\ar@{}+<0pt,-20pt>*{.}\ar@{}+<0pt,-26pt>*{.}\ar@{}+<0pt,-32pt>*{.}&&*+={\bullet}\ar@{}+<-10pt,0pt>*{_{I_2}}
\ar@{}+<6pt,-6pt>*{.}\ar@{}+<10pt,-10pt>*{.}\ar@{}+<14pt,-14pt>*{.}
\ar@{}+<-6pt,-6pt>*{.}\ar@{}+<-10pt,-10pt>*{.}\ar@{}+<-14pt,-14pt>*{.}
&&&\\
&&&&&&&
}
$$
\caption{}\label{fig1}
\end{figure}

\vspace{3mm}

Here the copies of $R$ form the left (vertical) boarder of this diagram. Furthermore the line consisting of
copies of $I_{\fty}$ goes in southeastern direction and forms another boarder. Note that starting with a copy of
$R$ and moving in the northeastern direction along the coray $R\gets I_1\gets I_2\gets\dots$ of irreducible maps
we obtain $I_{\fty}$ as the inverse limit (i.e. the intersection) of the corresponding inverse system.

\section{Free realizations of pp-formulae and patterns}\label{S-free}

We will briefly recall some notions of model theory of modules. In the first part of this section $S$ will denote an
arbitrary commutative ring. For more explanations the reader is referred to \cite[Ch. 1--4]{Preb}.

A \emph{positive-primitive formula} $\phi(v)$ in one free variable $v$ is an existentially quantified formula
$\ex\, \ov v\, (\ov v A= v \bar b)$, where $\ov v= (v_1, \dots, v_k)$ is a tuple of bounded variables, $A$ is a
$k\times n$ matrix over $S$ and $\ov b$ is a row of elements of $S$ of length $n$. If $M$ is an $S$-module and
$m\in M$ we say that \emph{$M$ satisfies $\phi(m)$}, written $M\models \phi(m)$, if there exists a tuple
$\ov m= (m_1, \dots, m_k)$ of elements from $M$ such that $\ov m A= m\bar b$.

This pp-formula claims the decidability, on $m$ in $M$, of the system of linear equations given by the matrix
$(A|\ov b)$. For instance if $r\in S$ then $vr= 0$ is the \emph{annihilator formula}, and
$r\mid v\doteq \ex\, w\, (wr=v)$ is the \emph{divisibility formula}.

Since $S$ is commutative, the set $\phi(M)= \{m\in M\mid M\ms \phi(m)\}$ is a submodule of $M$. For instance
$(r\mid v)(M)= Mr$, and $(vr=0)(M)$ is the kernel of right multiplication by $r$.

If $\phi$ and $\psi$ are pp-formulae we say the $\phi$ \emph{implies} $\psi$, written $\phi\to \psi$ or
$\phi\leq \psi$, if $\phi(M)\seq \psi(M)$ for any module $M$. These formulas are said to be \emph{equivalent} if
both $\phi\leq \psi$ and $\psi\leq \phi$ hold. The equivalence classes of pp-formulae form a modular lattice $L$ of
pp-formulae over $S$. The meet in this lattice is the conjunction of pp-formulae, and the sum is the pp-formula
$(\phi+ \psi)(v)\doteq \ex\, w\, \phi(v)\wg \psi(v-w)$.

Each pp-formula $\phi$ has a finitely generated \emph{free realization}, i.e. a finitely generated module $M$
pointed at an element $m$ such that 1) $M\ms \phi(m)$ and 2) if $M\ms \psi(m)$ for some pp-formulae $\psi$ then
$\phi\leq \psi$. For instance the pointed module $(S,r)$ is a free realization of the divisibility formula $r\mid v$,
and the module $S/rS$ pointed at $\ov 1$ is a free realization of the annihilator formula $vr=0$.

Suppose $(M,m)$ is a free realization of a pp-formula $\phi$ and $(N,n)$ is a free realization of a pp-formula
$\psi$. Then $\psi\leq \phi$ iff there exists a \emph{pointed morphism}, i.e. a morphism $f: M\to N$ of $S$-modules
such that $f(m)= n$.

Algebraically pp-formulae in one free variable can be defined as finitely generated subfunctors of the
functor $\Hom(S,-)$ from the category of $S$-modules to the category of abelian groups. For instance the
annihilator formula $vr=0$ corresponds to the functor $\Hom(S/rS,-)$, and the divisibility formula $r\mid v$
gives rise to the functor $G$ such that $G(M)= Mr$ for each module $M$.

Recall (see \cite[Sect. 3.4]{Preb}) that a class of $S$-modules $\mD$ is  said to be \emph{definable} if it is closed
with respect to direct products, direct limits and pure submodules. In fact this class consists of models of a
uniquely determined first order theory $T$.

The notion of the lattice of pp-formulae can be relativized to any definable class. Namely for pp-formulas
$\phi, \psi$ we set $\phi\leq_T \psi$ if $\phi(M)\seq \psi(M)$ for any module in $\mD$. The corresponding lattice of
pp-formulae $L_T$ will be the factor of the lattice $L$ of all pp-formulae. For instance, if $\mD_M$ is the smallest
definable class containing a module $M$, then we obtain the lattice $L_M$ of \emph{pp-definable submodules} of $M$.

The class of $\CM$-modules over $R$ is clearly definable, hence we obtain the theory $T_{CM}$ of Cohen-Macaulay
modules over $R$, and could relativize the above notions to this theory. For instance $L_{CM}$ will denote the
lattice of pp-formulae in this theory, which is a factor of the lattice of pp-formulae by a certain
congruence relation. In fact it is not difficult to choose a canonical representative in each equivalence class.

Namely let $\phi$ be a pp-formula freely realized by a pair $m\in M$, where $M$ is a finitely generated $R$-module.
If $\ov M$ is a factor of $M$ by its $y$-torsion submodule, it is a $\CM$-module. Let $\ov m$ denote the image
of $m$ in $\ov M$ and let a pp-formula $\phi_{CM}$ generate the pp-type of $\ov m$ in $\ov M$. Clearly we have the
implication $\phi_{CM}\to \phi$, and both formulas are equivalent in $T_{CM}$.

Thus the lattice $L_{CM}$ of pp-formulae in $T_{CM}$ can be defined as a lattice of pointed finitely generated
$\CM$-modules. Namely if $m\in M$ represents a pp-formula $\phi$ and $n\in N$ gives $\psi$, then the
pair $(m,n)\in M\oplus N$ represents $\phi+ \psi$. Their conjunction $\phi\wg \psi$ is given by the factor
of the following pushout module $(K,k)$ by its $y$-torsion submodule.

$$
\xymatrix@C=14pt@R=20pt{%
(R,1)\ar[r]\ar[d]&(M,m)\ar@{-->}[d]&\\
(N,n)\ar@{-->}[r]&(K,k)\ar[r]^{\pi}& (\ov K,\ov k)
}
$$

We will often identify pp-formulas in $L_{CM}$ with their free realizations in finitely generated $\CM$-modules and
refer to them as \emph{$\CM$-formulae}.

Now we consider the so-called patterns of pointed finitely generated $\CM$-modules over $R$. The term is borrowed
from Ringel \cite[Sect. 3]{Rin80} where he calculates the hammock functors: the traces $\Hom(N,-)$ of simple regular
modules $N$ in indecomposable finite dimensional modules over tame hereditary finite dimensional algebras.

A \emph{pointed $\CM$-module} is a pair $(M,m)$, where $M$ is a $\CM$-module and $m$ is an element of $M$:
we will often write this as $m\in M$. This pair can be also thought of as a morphism $R\to M$ sending $1$ to $m$.
A \emph{pointed morphism} (or just morphism) of pointed modules $(M,m)$ and $(N,n)$ is a morphism $f: M\to N$
of $R$-modules such that $f(m)= n$.

Let $(M,m)$ be an indecomposable finitely generated pointed $\CM$-module. We will introduce a \emph{$\CM$-pattern}
(or just pattern) $P$ of $(M,m)$, which is a partially ordered set, as follows. The elements of $P$ will be
equivalence classes of pointed morphisms from $(M,m)$ to indecomposable finitely generated $\CM$-modules.

Let $f: (M,m)\to (N,n)$ and $g: (M,m)\to (K,k)$ be pointed morphisms in $P$. We set $f\geq g$ if there exists
a morphism $h: N\to K$ such that $h(n)= k$. Clearly this relation is reflexive and transitive. To make it
anti-symmetric we factor $P$ by the equivalence relation $f\sim g$ if $f\geq g$ and $g\geq f$.

Usually the structure of $P$ is quite involved. The following example of a pattern will be of crucial importance
for this paper.

\begin{prop}\label{patt}
Figure \ref{fig2} shows the pattern of the pointed module $x\in I_{\fty}$.
\end{prop}

\begin{figure}
$$
\vcenter{%
\xymatrix@C=20pt@R=16pt{%
&&&&*+={\bullet}\ar@{-}[dr]\ar@{}+<0pt,10pt>*{_{x\in I_{\fty}}}&&\\
&&&&&*+={\bullet}\ar@{-}[dr]\ar@{}+<12pt,10pt>*{_{xy\in I_{\fty}}}&&\\
&&*+={\bullet}\ar@{-}[dr]\ar@{-}[dl]
\ar@{}+<-12pt,8pt>*{_{x\in I_2}}\ar@{}+<6pt,6pt>*{.}\ar@{}+<10pt,10pt>*{.}\ar@{}+<14pt,14pt>*{.}&&&&
*+={\bullet}\ar@{}+<12pt,10pt>*{_{xy^2\in I_{\fty}}}
\ar@{}+<6pt,-6pt>*{.}\ar@{}+<10pt,-10pt>*{.}\ar@{}+<14pt,-14pt>*{.}\\
&*+={\bullet}\ar@{-}[dr]\ar@{-}[dl]\ar@{}+<-12pt,8pt>*{_{x\in I_1}}&&
*+={\bullet}\ar@{-}[dr]\ar@{-}[dl]\ar@{}+<-6pt,8pt>*{_{xy\in I_3}}
\ar@{}+<6pt,6pt>*{.}\ar@{}+<10pt,10pt>*{.}\ar@{}+<14pt,14pt>*{.}&&&\\
*+={\bullet}\ar@{-}[dr]\ar@{}+<-12pt,8pt>*{_{x\in R}}&&
*+={\bullet}\ar@{-}[dr]\ar@{-}[dl]\ar@{}+<-6pt,8pt>*{_{xy\in I_2}}&&
*+={\bullet}\ar@{-}[dl]\ar@{}+<-6pt,8pt>*{_{xy^2\in I_4}}
\ar@{}+<6pt,6pt>*{.}\ar@{}+<10pt,10pt>*{.}\ar@{}+<14pt,14pt>*{.}
\ar@{}+<6pt,-6pt>*{.}\ar@{}+<10pt,-10pt>*{.}\ar@{}+<14pt,-14pt>*{.}&&\\
&*+={\bullet}\ar@{-}[dr]\ar@{-}[dl]\ar@{}+<-6pt,8pt>*{_{xy\in I_1}}&&
*+={\bullet}\ar@{-}[dl]\ar@{}+<-6pt,8pt>*{_{xy^2\in I_3}}
\ar@{}+<6pt,-6pt>*{.}\ar@{}+<10pt,-10pt>*{.}\ar@{}+<14pt,-14pt>*{.}&&&\\
*+={\bullet}\ar@{-}[dr]\ar@{}+<-6pt,8pt>*{_{xy\in R}}&&
*+={\bullet}\ar@{-}[dl]\ar@{}+<-6pt,8pt>*{_{xy^2\in I_2}}\ar@{}+<6pt,-6pt>*{.}
\ar@{}+<10pt,-10pt>*{.}\ar@{}+<14pt,-14pt>*{.}&&&&\\
&*+={\bullet}\ar@{-}[dl]\ar@{}+<-6pt,8pt>*{_{xy^2\in I_1}}
\ar@{}+<6pt,-6pt>*{.}\ar@{}+<10pt,-10pt>*{.}\ar@{}+<14pt,-14pt>*{.}&&&&&\\
*+={\bullet}\ar@{}+<-8pt,8pt>*{_{xy^2\in R}}
\ar@{}+<6pt,-6pt>*{.}\ar@{}+<10pt,-10pt>*{.}\ar@{}+<14pt,-14pt>*{.}&&&&&&\\
}}
$$
\caption{}\label{fig2}
\end{figure}

Observe that this diagram is just a filter in the $\AR$-quiver of $R$ - see Figure \ref{fig1}.

\begin{proof}
Every nonzero morphism $f$ from $I_{\fty}$ to an indecomposable finitely generated $\CM$-module $N$ sends
$x$, up to a multiplicative unit in $F[[y]]$, to $xy^i$, hence $f$ is given by multiplication by $y^i$. Clearly
multiplication by this unit does not change the equivalence class of $f$.

The remaining part is by inspection. For instance $xy\in I_2\geq xy\in I_1$ by inclusion, and
$xy\in I_2\geq xy^2\in I_3$ via multiplication by $y$.
\end{proof}

Thus on the diagram we see the traces of the functor $\Hom(I_{\fty},-)$ on indecomposable finitely generated
$\CM$-modules. The pattern of another pointed module $x\in R$ consists of points on Figure \ref{fig2} located
below (and including) this module. 

If $\phi$ and $\psi$ are pp-formulae then $[\phi\wg \psi, \phi]$ denotes the corresponding interval in the
lattice of pp-formulae (or some relativized version of it). We will often refer to this interval by saying
'$\phi$ over $\psi$'. For instance the pair of pp-formulas $\psi< \phi$ is said to be \emph{minimal} in a theory $T$
if the interval $[\psi, \phi]_T$ is simple.

\begin{lemma}\label{dist}
The interval $[v=0, vx=0]$ in the lattice of $\CM$ formulae over $R$ is distributive.
\end{lemma}
\begin{proof}
By Remark \ref{dense} and similarly to \cite[Prop. 4.4]{P-P} it suffices to prove that, when evaluated on any
indecomposable finitely generated $\CM$ module $M$, this interval is a chain. But this is obvious, because
every cyclic submodule of $(vx=0)(M)$ is generated by $xy^i$ for some $i$, and such submodules are linearly
ordered by inclusion.
\end{proof}

Now we are in a position to show that the diagram on Figure \ref{fig2} describes the above interval.

\begin{prop}\label{int-desc}
1) Every $\CM$-formula in the interval $[v=0, vx=0]$ is equivalent to a finite sum of marked on Figure \ref{fig2}
formulas.

2) This sum is free in the following sense. If $\phi_i$ and $\psi_j$ are finite sets of marked formulae,
then $\sum_i \phi_i$ implies $\sum_j \psi_j$ if and only if for every $i$ there exists $j$ such that
$\phi_i\leq \psi_j$.

3) The intersection of marked formulae equals the marked intersection, but this is not true for sums.
\end{prop}
\begin{proof}
1) Because $x\in I_{\fty}$ is a free realization of the annihilator formula $vx=0$, every $\CM$-formula $\psi$
below $vx=0$ corresponds to a pointed morphism from $x\in I_{\fty}$ to a finitely generated $\CM$-module $N$.
Decomposing $N$ into a direct sum of indecomposables we conclude that $\psi$ is equivalent to a finite sum of
$\CM$ formulae $\psi_j$ given by pointed morphisms from $I_{\fty}$ to indecomposable finitely generated
$\CM$-modules.

By the description of such morphisms ($x$ goes to $xy^i$) we derive that $\psi_j$ is equivalent to a marked
formula on the diagram.

2) Clearly it suffices to consider the implication $\phi\to \psi_1+ \dots+ \psi_m$ for $\CM$-formulas on the diagram.
Let the pair $(N,n)$ on the diagram represents $\phi$. Evaluating on $N$, by uniseriality, we conclude that
$\phi(N)\seq \psi_j(N)$ for some $j$. But then $\phi$ implies $\psi_j$ by the definition of a free realization.

3) Let $\phi$ and $\psi$ be incomparable $\CM$-formulae marked on the diagram. By what we have proved the
$\CM$-formula $\phi\wg \psi$ is equivalent to a sum of formulae on the diagram which lie below $\phi$
and $\psi$. But from the diagram we see that there is a unique largest element in this set, hence it must be
equal to the conjunction.
\end{proof}

Now the $\CM$-formulas in the interval $[v=0, vx=0]$ are easily visualized. For instance, the interval
$x\in I_{\fty}$ over $xy\in I_{\fty}$ in $L_{CM}$ is the following chain of order type $\om+ 1$.

$$
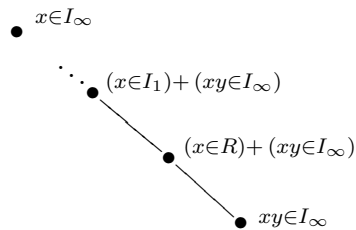
\vcenter{%
\xymatrix@C=20pt@R=16pt{%
*+={\bullet}\ar@{}+<18pt,6pt>*{_{x\in I_{\fty}}}\\
&*+={\bullet}\ar@{}+<38pt,4pt>*{_{(x\in I_1)+\, (xy\in I_{\fty})}}\ar@{-}[rd]
\ar@{}+<-4pt,4pt>*{.}\ar@{}+<-8pt,7pt>*{.}\ar@{}+<-12pt,10pt>*{.}\\
&&*+={\bullet}\ar@{}+<38pt,4pt>*{_{(x\in R)+\, (xy\in I_{\fty})}}\ar@{-}[rd]\\
&&&*+={\bullet}\ar@{}+<20pt,2pt>*{_{xy\in I_{\fty}}}\\
}}
$$

\vspace{3mm}

In particular the pair $(x\in R)+ (xy\in I_{\fty})$ over $xy\in I_{\fty}$ is minimal in $T_{CM}$.

Similarly the interval $x\in I_1$ over $x\in R$ in $L_{CM}$ is the following chain of order type $1+ \om^*$.

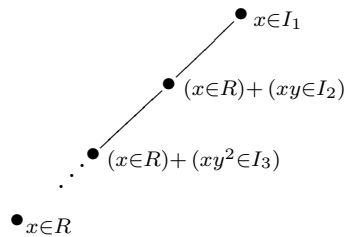
$$
\vcenter{%
\xymatrix@C=20pt@R=18pt{%
&&&*+={\bullet}\ar@{}+<14pt,-2pt>*{_{x\in I_1}}\ar@{-}[ld]\\
&&*+={\bullet}\ar@{}+<36pt,-2pt>*{_{(x\in R)+\, (xy\in I_2)}}\ar@{-}[ld]\\
&*+={\bullet}\ar@{}+<38pt,-2pt>*{_{(x\in R)+\, (xy^2\in I_3)}}
\ar@{}+<-4pt,-4pt>*{.}\ar@{}+<-8pt,-8pt>*{.}\ar@{}+<-12pt,-12pt>*{.}\\
*+={\bullet}\ar@{}+<12pt,-2pt>*{_{x\in R}}\\
}}
$$

\vspace{3mm}

Furthermore the pair $x\in I_1$ over $(x\in R)+ (xy\in I_2)$ is minimal in $T_{CM}$ and corresponds to the
left almost split map\rule{0pt}{6mm} $I_1\xr{\bsm 1\\y\esm} R\oplus I_2$. Also the pair $x\in R$ over
$xy\in I_1$ is minimal in $T_{CM}$ and corresponds to the irreducible map $R\xr{y} I_1$ in the category
of finitely generated $\CM$-modules.

Note that the whole lattice of $\CM$-formulae over $R$ is more complex and involves rather an unintelligent
picture; for instance it is not distributive. Namely otherwise $R$ would have a distributive lattice of
ideals. Since $R$ is local it would yield that $R$ is a valuation domain, a contradiction.

\section{The Ziegler spectrum}\label{S-zig}

First we recall few more definitions from model theory of modules. A submodule $M$ of a module $N$ is said to
be \emph{pure} if, for every $m\in M$ and a pp-formula $\phi$, from $N\ms \phi(m)$ it follows that $N\ms \phi(n)$;
and this definition is obviously extended to embeddings of modules.

A module $M$ is said to be \emph{pure injective}, if it injective with respect to pure monomorphisms of modules. An
equivalent requirement is that $M$ splits any pure embedding. For instance each injective module is pure injective.
Furthermore by Ringel (see \cite[L. 4.2.8]{Preb}) each module which is linearly compact over its  endomorphism ring
is pure injective. Because $R$ is a complete noetherian ring, it is linearly compact, hence pure injective. Since
linearly compactness respects extensions, each finitely generated $R$-module is also pure injective.

The \emph{Ziegler spectrum} of a ring $R$, $\Zg_R$, is a topological space whose points are isomorphism types
of indecomposable pure injective modules. The topology on this space is given by pairs $(\phi/\psi)$ of
pp-formulae. Here the open set $(\phi/\psi)$ consists of points $M$ such that there exists $m\in M$ satisfying
$\phi$ but not $\psi$. In this case we will say that $M$ \emph{opens} the corresponding interval
$[\phi\wg \psi, \phi]$.

It follows from Ziegler (see \cite[Thm. 5.1.22]{Preb}) that $(\phi/\psi)$ is a quasi-compact open set, in particular
$\Zg_R$ is a quasi-compact space, which is often non-Hausdorff. More properties of Ziegler spectrum can be found in
\cite[Ch. 5]{Preb}

Since the class of $\CM$-modules is definable, one could talk about the theory $T_{CM}$ of $\CM$-modules, therefore
about the closed subset of $\Zg_R$ consisting of indecomposable pure injective $\CM$-modules, with the induced
topology. We will call this topological space the \emph{$\CM$-part of the Ziegler spectrum}, $\ZCM_R$; and
investigating this space is the main objective of this paper. For instance each basic open set $(\phi/\psi)_T$
is compact so as the whole space $\ZCM$.

As we have already noticed every indecomposable finitely generated $\CM$-module over $R$ is a point in $\ZCM$.
Furthermore, since each $\CM$-module over $R$ is a union of its finitely generated submodule we obtain the following.

\begin{remark}\label{dense}
The finitely generated points are dense in $\ZCM$.
\end{remark}

We add more points to this list.

\begin{remark}\label{more}
$\wt R$, $Q$ and the Laurent power series field $G= F((y))$ are points in $\ZCM$.
\end{remark}

Here are the shapes of these modules.

\vspace{3mm}

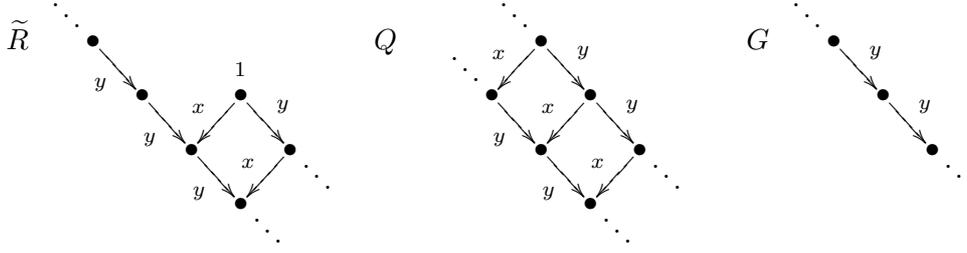
\begin{center}
$
\wt R\hspace{3mm}
\xymatrix@C=10pt@R=12pt{%
*+={\bullet}\ar[dr]_y\ar@{}+<-6pt,6pt>*{.}\ar@{}+<-10pt,10pt>*{.}\ar@{}+<-14pt,14pt>*{.}&&&&\\
&*+={\bullet}\ar[dr]_y&&*+={\bullet}\ar[dl]_x\ar[dr]^y\ar@{}+<0pt,10pt>*{_1}&\\
&&*+={\bullet}\ar[dr]_y&&*+={\bullet}\ar[dl]_x\ar@{}+<6pt,-6pt>*{.}\ar@{}+<10pt,-10pt>*{.}\ar@{}+<14pt,-14pt>*{.}\\
&&&*+={\bullet}\ar@{}+<6pt,-6pt>*{.}\ar@{}+<10pt,-10pt>*{.}\ar@{}+<14pt,-14pt>*{.}&
}
$
\hspace{3mm}
$
Q\hspace{7mm}
\xymatrix@C=10pt@R=12pt{%
&*+={\bullet}\ar[dl]_x\ar[dr]^y\ar@{}+<-6pt,6pt>*{.}\ar@{}+<-10pt,10pt>*{.}\ar@{}+<-14pt,14pt>*{.}&&&&\\
*+={\bullet}\ar[dr]_y\ar@{}+<-6pt,6pt>*{.}\ar@{}+<-10pt,10pt>*{.}\ar@{}+<-14pt,14pt>*{.}&&
*+={\bullet}\ar[dl]_x\ar[dr]^y&&\\
&*+={\bullet}\ar[dr]_y&&*+={\bullet}\ar[dl]_x\ar@{}+<6pt,-6pt>*{.}\ar@{}+<10pt,-10pt>*{.}\ar@{}+<14pt,-14pt>*{.}\\
&&*+={\bullet}\ar@{}+<6pt,-6pt>*{.}\ar@{}+<10pt,-10pt>*{.}\ar@{}+<14pt,-14pt>*{.}&&
}
$
\hspace{-3mm}
$
G\hspace{3mm}
\xymatrix@C=10pt@R=12pt{%
*+={\bullet}\ar[dr]^y\ar@{}+<-6pt,6pt>*{.}\ar@{}+<-10pt,10pt>*{.}\ar@{}+<-14pt,14pt>*{.}&&&\\
&*+={\bullet}\ar[dr]^y&&\\
&&*+={\bullet}\ar@{}+<6pt,-6pt>*{.}\ar@{}+<10pt,-10pt>*{.}\ar@{}+<14pt,-14pt>*{.}&\\
}
$
\end{center}

\vspace{3mm}

\begin{proof}
Clearly all these modules are $\CM$.

Because $Q$ is injective and uniform it is indecomposable. Also $G= F((y))$ is annihilated by $x$ and is injective
and indecomposable as an $R/xR= F[[y]]$-module, hence pure injective over $R$.

Being uniform, the module $\wt R$ is indecomposable. We check that it is linearly compact, hence pure injective.
Consider the following filtration of $\wt R$ considered as a submodule of $Q$:
$0\subset \wt R x\subset Qx\subset \wt R$. Here the module $\wt Rx\cong F[[y]]$ is linearly compact so as
$\wt R/Qx\cong F[[y]]$. Further the same is true for $Qx/\wt Rx$ which is the Pr\"ufer $F[[y]]$-module.
Since linearly compactness respects extensions, we obtain the desired.
\end{proof}

Note that the filtration on $\wt R$ is given by pp-formulae, for instance $Qx= (vx=0)(\wt R)$.

Suppose that $M$ is a pure injective module and $m\in M$. The set of all pp-formulae $\phi(v)$ such that
$m\in \phi(M)$ is said to be a \emph{pp-type} of $m$ in $M$, written $pp_M(m)$. This set of formulae can
be also described as a filter in the lattice of pp-formulae, i.e. a subset of $L$ which is upward closed
and closed with respect to finite conjunctions; hence as a filter of finitely generated subfunctors of the
functor $\Hom(R,-)$. If $M$ is pure injective indecomposable and $m$ is nonzero, then $p$ is said to be
\emph{indecomposable}.

We will assign to a pp-type $p$ its \emph{positive part} $p^+$ consisting of pp-formulas in $p$, and its
\emph{negative part} $p^-$ which consists of pp-formulae not in $p$. If $\psi< \phi$  are pp-formulae,
then $p$ defines a \emph{cut} on the interval $[\psi, \phi]$ whose upper part is $p^+\cap [\psi, \phi]$,
and lower part is the intersection of $p^-$ with this interval. If $\phi\in p^+$ and $\psi\in p^-$, i.e.
the corresponding cut is nontrivial, then we say that \emph{$p$ opens} this interval.

It follows from Ziegler that an indecomposable pure injective module $M$ is uniquely determined by the pp-type of
any its nonzero element. Furthermore, see \cite[Thm. 5.1.24]{Preb}, to recover $M$ it suffices to know any arbitrary
small piece of local information on $p$, i.e. a nontrivial cut defined by $p$ on any interval of the lattice of
pp-formulae. Also, see \cite[Thm. 4.4]{Zie}, there is a useful criterion for checking when $p$ is indecomposable.
Of course all these relativizes to any theory of $R$-modules, in particular to the theory of $\CM$-modules.

We claim that the modules mentioned in Remark \ref{more} are the only remaining points in $\ZCM$.

\begin{theorem}\label{class-point}
The modules $\wt R$, $Q$ and $G$ are the only non-finitely generated points in the Cohen-Macaulay part of the Ziegler
spectrum of $R$.
\end{theorem}
\begin{proof}
Let $N$ be an indecomposable pure injective $\CM$-module over $R$. If $N$ is annihilated by $x$, it is an
indecomposable pure injective $R/xR\cong F[[y]]$-module without $y$-torsion. The description of the Ziegler
spectrum of this valuation domain is well known - see \cite[Sect. 5.2.1]{Preb}. It follows that either
$N\cong F[[y]]$ (the adic point), hence $N\cong I_{\fty}$; or $N\cong F((y))$ (the generic point), therefore
$N\cong G$.

Thus we may assume that $Nx\neq 0$. Choose $0\neq n\in Nx$ and let $p$ be the pp-type of $n$ in $N$. In
particular $p$ is uniquely determined by its $\CM$-part, and the formula $vx= 0$ (even $x\mid v$!) belongs
to $p$. It follows that $p$ defines a nontrivial cut on the interval $[v=0, vx= 0]$ on Figure \ref{fig2}.
Furthermore the isomorphism type of $N$ is uniquely determined by this cut.

By Ziegler's criterion for indecomposability and distributivity (see Lemma \ref{dist}) we conclude that for any
$\CM$ formulas $\phi, \psi\in p^-$ marked on the diagram, their sum $\phi+ \psi$ is also not in $p$. Thus the
indecomposable pp-type $p$ is uniquely determined by a filter of formulas on the diagram, i.e. by an upward closed
and closed with respect to conjunctions set of marked formulae (recall that the sums are free, hence hidden on
the diagram).

Then either $p$ is finitely generated (i.e. the corresponding filter is principal), hence is realized in an
indecomposable finitely generated $\CM$-module, or is generated by a line going in the southeastern direction
on the diagram, or includes all nonzero $\CM$ formulae from the interval. Comparing pp-types we obtain the
possibilities shown on Figure \ref{fig3}.

\vspace{3mm}

\begin{figure}
$$
\vcenter{%
\xymatrix@C=20pt@R=16pt{%
&&&&*+={\bullet}\ar@{-}[dr]\ar@{}+<0pt,10pt>*{_{x\in I_{\fty}}}&&\\
&&&&&*+={\bullet}\ar@{-}[dr]\ar@{}+<12pt,10pt>*{_{xy\in I_{\fty}}}&&\\
&&*+={\bullet}\ar@{-}[dr]\ar@{-}[dl]
\ar@{}+<-12pt,8pt>*{_{x\in I_2}}\ar@{}+<6pt,6pt>*{.}\ar@{}+<10pt,10pt>*{.}\ar@{}+<14pt,14pt>*{.}&&&&
*+={\bullet}\ar@{}+<12pt,10pt>*{_{xy^2\in I_{\fty}}}
\ar@{}+<6pt,-4pt>*{.}\ar@{}+<10pt,-8pt>*{.}\ar@{}+<14pt,-12pt>*{.}\\
&*+={\bullet}\ar@{-}[dr]\ar@{-}[dl]\ar@{}+<-12pt,8pt>*{_{x\in I_1}}&&
*+={\bullet}\ar@{-}[dr]\ar@{-}[dl]\ar@{}+<-6pt,8pt>*{_{xy\in I_3}}
\ar@{}+<6pt,6pt>*{.}\ar@{}+<10pt,10pt>*{.}\ar@{}+<14pt,14pt>*{.}&&&&*+={\bullet}\ar@{}+<16pt,0pt>*{_{1\in G}}\\
*+={\bullet}\ar@{-}[dr]\ar@{}+<-12pt,8pt>*{_{x\in R}}&&
*+={\bullet}\ar@{-}[dr]\ar@{-}[dl]\ar@{}+<-6pt,8pt>*{_{xy\in I_2}}&&
*+={\bullet}\ar@{-}[dl]\ar@{}+<-6pt,8pt>*{_{xy^2\in I_4}}\ar@{}+<6pt,6pt>*{.}\ar@{}+<10pt,10pt>*{.}\ar@{}+<14pt,14pt>*{.}
\ar@{}+<6pt,-6pt>*{.}\ar@{}+<10pt,-10pt>*{.}\ar@{}+<14pt,-14pt>*{.}&&\\
&*+={\bullet}\ar@{-}[dr]\ar@{-}[dl]\ar@{}+<-6pt,8pt>*{_{xy\in I_1}}&&
*+={\bullet}\ar@{-}[dl]\ar@{}+<-6pt,8pt>*{_{xy^2\in I_3}}
\ar@{}+<6pt,-6pt>*{.}\ar@{}+<10pt,-10pt>*{.}\ar@{}+<14pt,-14pt>*{.}&&
*+={\bullet}\ar@{}+<22pt,0pt>*{_{xy^{-2}\in \wt R}}&\\
*+={\bullet}\ar@{-}[dr]\ar@{}+<-6pt,8pt>*{_{xy\in R}}&&
*+={\bullet}\ar@{-}[dl]\ar@{}+<-6pt,8pt>*{_{xy^2\in I_2}}\ar@{}+<6pt,-6pt>*{.}
\ar@{}+<10pt,-10pt>*{.}\ar@{}+<14pt,-14pt>*{.}&&*+={\bullet}\ar@{}+<23pt,0pt>*{_{xy^{-1}\in \wt R}}&&\\
&*+={\bullet}\ar@{-}[dl]\ar@{}+<-6pt,8pt>*{_{xy^2\in I_1}}
\ar@{}+<6pt,-6pt>*{.}\ar@{}+<10pt,-10pt>*{.}\ar@{}+<14pt,-14pt>*{.}&&*+={\bullet}\ar@{}+<16pt,0pt>*{_{x\in \wt R}}&&&\\
*+={\bullet}\ar@{}+<-8pt,8pt>*{_{xy^2\in R}}
\ar@{}+<6pt,-6pt>*{.}\ar@{}+<10pt,-10pt>*{.}\ar@{}+<14pt,-14pt>*{.}&&*+={\bullet}\ar@{}+<20pt,0pt>*{_{xy\in \wt R}}&&&&\\
&*+={\bullet}\ar@{}+<22pt,0pt>*{_{xy^2\in \wt R}}\\
*+={\bullet}\ar@{}+<16pt,0pt>*{_{x\in Q}}
}}
$$
\caption{}\label{fig3}
\end{figure}

\vspace{3mm}

For instance the (indecomposable) pp-type $p$ defined by the ray $(R,x)\xr{y} (I_1, xy)\xr{y} (I_2, xy^2)\xr{y}\dots$
coincides with the pp-type of $x\in \wt R$, therefore $\wt R$ is the pure injective envelope of $p$. Similarly we
will obtain $\wt R $ as the direct limit of the parallel ray of irreducible morphisms $I_1\xr{y} I_2\xr{y}\dots$
starting with $x\in I_1$, where $xy^{-1}\in \wt R$ realizes the corresponding pp-type.

Further $G$ is the direct limit of the ray of irreducible morphisms $I_{\fty}\xr{y} I_{\fty}\xr{y}\dots$, where
$1\in G$ realizes $p$. Also $Q$ is the direct limit of the directed system $R\xr{y} R\xr{y}\dots$ heading downwards,
where the pp-type of $x\in Q$ equals $p$ and is critical over zero.
\end{proof}

To complete a description of $\ZCM$ it remains to describe the topology, i.e. to give a basis of open sets to each
point of this space. We will describe this open basis as a collection of basic open sets $(\phi/\psi)$ containing
this point. Evaluating we will also give a subset of the Ziegler spectrum corresponding to each open set
in this basis. In fact all this can be read off the diagram.

Namely it follows from Ziegler \cite[Thm. 4.9]{Zie} that, if $\phi\in p^+, \psi\in p^-$ for pp-formulae
$\psi< \phi$ and an indecomposable pp-type $p$, then the basis of open sets for the pure injective envelope of
$p$ can be chosen among pairs $(\phi'/\psi')$, where $\psi\leq \psi'< \phi'\leq \phi$ and $\psi'\in p^-$,
$\phi'\in p^+$.

\begin{remark}\label{in}
Every finitely generated point $I_n$, $n< \fty$ is isolated by a minimal pair in $\ZCM$.
\end{remark}
\begin{proof}
By inspection of Figure \ref{fig2} using $\AR$-sequences.

For instance $I_1$ is isolated by the minimal pair $x\in I_1$ over $(x\in R)+ (xy\in I_2)$ coming from the
corresponding left almost split map. Also $R$ is isolated by the minimal pair $x\in R$ over $xy\in I_1$ given
by an irreducible morphism (multiplication by $y$).
\end{proof}

For $I_{\fty}$ one should work harder.

\begin{lemma}\label{fty}
Let $O_m$ denote the (cofinite) set of finite points $I_n$, $n\geq m$. The basis of open sets for $I_{\fty}$ is given
by sets $O_m\cup I_{\fty}$, $m= 1, 2, \dots$. For instance $I_{\fty}$ is not isolated.
\end{lemma}
\begin{proof}
Since $x\in I_{\fty}$ satisfies the formula $vx= 0$, the basis of open sets for $I_{\fty}$ can be chosen within
the interval $[x=0, vx=0]$, i.e. it is visible on Figure \ref{fig2}. For instance $x\in I_{\fty}$ over
$(xy\in I_{\fty})+ (x\in I_m)$ is such a basis.

By evaluating we see that the open set defined by this pair is $O_{m+1}\cup I_{\fty}$.
\end{proof}

Recall (see \cite[p. 254]{Preb}) that an indecomposable pp-type $p$ is said to be \emph{neg-isolated} if
there a pp-formula $\psi\in p^-$ such that $p$ is maximal among pp-types containing $\psi$ in its negative part.
If $M$ is an indecomposable pure injective module realizing a neg-isolated pp-type $p$, then $M$ is called
\emph{neg-isolated}; this notion does not depend on the choice of $p$. For instance from the above description
it follows that $I_{\fty}$ is not neg-isolated.

Now we are in a position to describe isolated points.

\begin{cor}\label{isol}
The isolated points in $\ZCM$ of $R$ are exactly modules $I_n$, $0\leq n< \fty$.
\end{cor}
\begin{proof}
Since $G$ is the direct limit of copies of $I_{\fty}$ (see Figure \ref{fig3}) and the space $\ZCM$ is compact,
this module cannot be isolated. By a similar reason neither $\wt R$ nor $Q$ is isolated. As we have already
seen $I_{\fty}$ is also non-isolated. It remains to apply Remark \ref{in}.
\end{proof}

Here we meet the first peculiarity. For finite dimensional algebras the existence of $\AR$-sequences implies
(see \cite[Cor. 5.3.37]{Preb}) that isolated points in the Ziegler spectrum are exactly indecomposable finite
dimensional modules. The original feeling was that these finite dimensional points should correspond to finitely
generated $\CM$-points in out setting. However this is not the case, because $I_{\fty}$ is finitely generated but
not isolated.

Now we deal with with $\wt R$.

\begin{lemma}\label{tR}
Let $O_m$ be as in Lemma \ref{fty}. The basis of open sets for $\wt R$ is given by sets $O_m\cup \wt R$,
$m= 1, 2, \dots$.
\end{lemma}
\begin{proof}
Consider $\wt R$ as a pointed module $x\in \wt R$, i.e. as the pure injective envelope of the pp-type $p$ of
$x\in \wt R$. From Figure \ref{fig3} we conclude that a basis of open sets for this point is given by the line
going from $x\in R$ in the southeastern direction, hence by pairs $xy^n\in I_n$ over $xy\in R$, $n= 1, 2, \dots$.
In particular this point is neg-isolated.

Evaluating this pair on $\ZCM$ we obtain the open set $O_n\cup \wt R$, as desired.
\end{proof}

Recall that the Cantor--Bendixson analysis on a compact topological space runs by a consecutive removal its
isolated points - see \cite[Sect. 5.3]{Preb} how it applies to the Ziegler spectrum. This way one obtains an
ordinal indexed descending chain of closed subspaces $T^{(\lam)}$. If $T^{(\lam)}\neq \emptyset$ and
$T^{(\lam+1)}= \emptyset$ for some ordinal $\lam$ then we say that the \emph{Cantor--Bendixson rank},
$\CB$-rank, of $T$ equals $\lam$. In this case each point $t\in T$ is assigned its $\CB$-rank, being the smallest
$\mu$ such that $t\in T^{(\mu)}\sm T^{(\mu+1)}$.

Thus on the first step of the $\CB$-analysis of $\ZCM$ we remove all isolated points $I_n$, $n< \fty$. The next
level is described in the following lemma.

\begin{lemma}\label{cb-1}
$I_{\fty}$ and $\wt R$ are the only points in $\ZCM$ of Cantor--Bendixson rank $1$.
\end{lemma}
\begin{proof}
It follows from Lemmas \ref{fty}, \ref{tR} that both points have $\CB$-rank 1. Furthermore $G$ is the direct limit
of copies of $I_{\fty}$, hence is in the closure of this point, in particular $G$ has $\CB$-rank at least 2.

Similarly $Q$ is the direct limit of copies of $\wt R$, therefore its $\CB$-rank exceeds $1$.
\end{proof}

The next level of the Cantor--Bendixson analysis is the last.

\begin{lemma}
The point $G$ has $\CB$-rank 2. Furthermore a basis of open sets for $G$ is obtained from $\ZCM$ by
removing $Q$ and finitely many finite points $I_n$, $n\leq m$.
\end{lemma}
\begin{proof}
Choosing $1\in G$, from Figure \ref{fig3} we conclude that a basis of open sets for $G$ is given by pairs
$xy^n\in I_{\fty}$ over $x\in I_m$, $m, n = 0, 1, \dots$.

The remaining part is by evaluation. For instance the finite point $I_k$ belongs to the above open set iff
$k> m+n$; and $I_{\fty}, \wt R$ belong to any of those open sets.

We have already seen that $\CB(G)\geq 2$. Since the pair $x\in I_{\fty}$ over $x\in R$ separates $G$ from $Q$ we
conclude that $\CB(G)= 2$.
\end{proof}

To complete the Cantor--Bendixson analysis it remains to deal with $Q$.

\begin{lemma}
The point $Q$ has $\CB$-rank 2. Furthermore a basis of open sets for $Q$ is obtained from $\ZCM$ by removing
$G$ and $I_{\fty}$.
\end{lemma}
\begin{proof}
Pointing $Q$ at $x$ from Figure \ref{fig3} we see that a basis of open sets for $Q$ is given by pp-pairs
$xy^n\in R$ over $x= 0$, hence this point is critical over zero in terminology of \cite{Her92}.

The remaining part is by evaluation, in particular $xy\in R$ over $x=0$ separates $Q$ from $G$ and $I_{\fty}$.
On the other hand none element in this basis separates $Q$ from any finite point $I_m$, hence $Q$ is in
the closure of each $I_m$.
\end{proof}

As a result we conclude on the value of $\CB$-rank.

\begin{theorem}\label{cb-val}
The Cantor--Bendixson rank of the Cohen--Macaulay part of the Ziegler spectrum of $R$ equals $2$.
\end{theorem}

From the above description of the topology it follows that the only closed points of $\ZCM$ are endofinite
points $G$ and $Q$, i.e. points of maximal $\CB$-rank.

Recall (see \cite[Exerc. 7.10]{J-L}) that over a finite dimensional algebra each pure injective module is
isomorphic to a direct summand of a direct product of finite dimensional modules. Here all looks differently.

\begin{lemma}\label{none}
None of the points $\wt R, Q$ and $G$ is isomorphic to a direct summand of a direct product of finitely generated
$\CM$-modules.
\end{lemma}
\begin{proof}
We will give a proof for $\wt R$, and the remaining cases are similar.

From Figure \ref{fig3} we conclude that the pp-type $p$ of $x\in \wt R$ is generated by formulas $xy^n\in I_n$,
$n= 1, 2, \dots$. Suppose that $\wt R$ is realized as a direct summand of a direct product of finitely generated
$\CM$-modules $M_i$, $i\in I$ such that the element $m= (m_i)_{i\in I}$ realizes $p$. If the $\CM$-formula $\phi_i$
generates the pp-type of $m_i$ in $M_i$, then $\phi_i$ will be below each pp-formula in $p$. Looking at
Figure \ref{fig3} we see that $\phi_i$ is equivalent to the zero formula, hence $m=0$, a contradiction.
\end{proof}

\section{Krull--Gabriel dimension and $m$-dimension}

In this section we will calculate the $m$-dimension of the lattice $L_{CM}$ of $\CM$-formulae over $R$. This is
the same as the Krull--Gabriel dimension of the definable category of Cohen--Macaulay $R$-modules
(see \cite[Sect. 13.2.2]{Preb}), but to avoid introducing many definitions we will not use this notion.

Suppose that $L$ is a lattice with top and bottom. Following \cite[Sect. 7.2]{Preb} by induction on ordinals we define
a sequence of congruence relations $\sim_{\lam}$, hence factor lattices $L_{\lam}= L/\sim_{\lam}$. Namely let
$\sim_0$ be the trivial relation, hence $L_0= L$. On limit steps $\lam$ we define
$\sim_{\lam}= \cup_{\mu< \lam} \sim_{\mu}$. Further, for a non-limit ordinal $\lam+ 1$, let the congruence
relation on $L_{\lam}$ collapses intervals of finite length, and let $\sim_{\lam+1}$ be the preimage of this
relation in $L$. If there is an ordinal $\lam$ such that the lattice $L_{\lam}$ is nontrivial (i.e. contains at
least two elements) and $L_{\lam+1}$ is a trivial lattice than we define the \emph{$m$-dimension} of $L$ to be
equal to $\lam$.

We will be interested in the case when $L$ is the lattice $L_{CM}$ of $\CM$-formulae over $R$. For instance on the
first step of the $m$-dimension analysis the simple interval $x\in R$ over $xy\in I_1$ on Figure \ref{fig3}
collapses in $L_1$. In general, the $m$-dimension analysis runs parallel to the Cantor-Bendixson analysis. Namely
(see \cite[Sect. 5.3.2]{Preb}) if, at each step of the Cantor-Bendixson analysis, each point is isolated by a
minimal pair, then the lattice $L_{\lam}$ coincides with the lattice of pp-formulae of the $\CB$ derivative
$T_{\CM}^{(\lam)}$ of the theory of $\CM$-modules.

From the analysis of the previous section it follows that this is the case for our singularity $R$, therefore
we obtain the following.

\begin{prop}\label{m-dim}
The $m$-dimension of the lattice of $\CM$-formulae over $R$ equals $2$.
\end{prop}

Recall (see \cite[Sect. 6.1]{Preb}) that a \emph{ring of definable scalars} of a module $M$ consists of
biendomorphisms of $M$ defined by pp-formulas $\phi(v,w)$ in two free variables, where the first variable describes
the domain, and the second sets the image of this map. In our case these objects are easily calculated.

\begin{prop}\label{def-scal}
1) The ring of definable scalars of the module $I_n$ coincides with the overring $R_n$ of $R$ generated (over $R$)
by $xy^{-n}$.

2) The ring of definable scalars of $I_{\fty}$ is the ring $F[[y]]$.

3) The ring of definable scalars of $\wt R$ equals $\wt R$.

4) The ring of definable scalars of $G$ is the field $F((y))$.

5) The ring of definable scalars of $Q$ equals $Q$.
\end{prop}

Note that in each case 1)--5) the module is isomorphic to its ring of definable scalars.

\begin{proof}
We will consider only the case 1), the remaining cases are by inspection.

Because $R$ is Gorenstein, it has a unique minimal overring which is easily checked to coincide with $R_1$.
Now looking at the diagram for $I_1$ from Section \ref{S-fg} we see that multiplication by $xy^{-1}$ defines a
biendomorphism of $I_1$: each element of this module when multiplied by $x$ is uniquely divisible by $y$.

Furthermore, because $I_1= \mm$ is isomorphic to $R_1$ (via the map $x\mapsto xy^{-1}$ and $y\mapsto 1$),
we conclude that the ring of definable scalars of $I_1$ coincides with $R_1$.

The ring $R_1$ is again Gorenstein, hence has a unique minimal overring $R_2$. Continuing this way we obtain the
desired.
\end{proof}

Recall that the lattice $L_{CM}$ of all $\CM$-formulas over $R$ is too large to be completely described. Its
first derivative is manageable.

\begin{prop}\label{t-prime}
The diagram on Figure \ref{fig4} represents the lattice $L_1$ of pp-formulas of the first derivative $T'$ of
the theory of Cohen--Macaulay modules over $R$. For instance this lattice is distributive.
\end{prop}

\begin{figure}[b]
$$
\def\objectstyle{\scriptstyle}
\xymatrix@C=14pt@R=12pt{%
&*+={\bullet}\ar@{-}[d]\ar@{}+<-15pt,2pt>*{v=v}&&\\
&*+={\bullet}\ar@{-}[dl]\ar@{-}[dr]&&\\
*+={\bullet}\ar@{-}[dr]\ar@{}+<-10pt,0pt>*{y\mid v}&&
*+={\bullet}\ar@{-}[dl]\ar@{}+<6pt,-6pt>*{.}\ar@{}+<10pt,-10pt>*{.}\ar@{}+<14pt,-14pt>*{.}&\\
&*+={\bullet}\ar@{-}[dl]\ar@{}+<6pt,-6pt>*{.}\ar@{}+<10pt,-10pt>*{.}\ar@{}+<14pt,-14pt>*{.}&&
*+={\bullet}\ar@{-}[dl]\ar@{}+<16pt,2pt>*{vx=0}\\
*+={\bullet}\ar@{}+<-12pt,0pt>*{y^2\mid v}\ar@{}+<6pt,-6pt>*{.}\ar@{}+<10pt,-10pt>*{.}\ar@{}+<14pt,-14pt>*{.}&&
*+={\bullet}\ar@{-}[dl]&\\
&*+={\bullet}\ar@{}+<0pt,-6pt>*{.}\ar@{}+<0pt,-10pt>*{.}\ar@{}+<0pt,-14pt>*{.}&&\\
&&&\\
&*+={\bullet}\ar@{-}[d]\ar@{}+<18pt,0pt>*{xy^{-1}\mid v}\ar@{}+<0pt,6pt>*{.}\ar@{}+<0pt,10pt>*{.}\ar@{}+<0pt,14pt>*{.}&&\\
&*+={\bullet}\ar@{-}[d]\ar@{-}[u]\ar@{}+<13pt,0pt>*{x\mid v}&&\\
&*+={\bullet}\ar@{}+<15pt,0pt>*{xy\mid v}\ar@{}+<0pt,-6pt>*{.}\ar@{}+<0pt,-10pt>*{.}\ar@{}+<0pt,-14pt>*{.}&&\\
&*+={\bullet}\ar@{}+<14pt,-2pt>*{v=0}&&\\
}
$$
\caption{}\label{fig4}
\end{figure}

\vspace{3mm}

\begin{proof}
It follows from the previous section that $T'$ is the closure, in the Ziegler spectrum, of points $I_{\fty}$ and
$\wt R$ of $\CB$-rank $1$. Since $I_{\fty}$ is definable in $\wt R$ it follows that $T'$ coincides with the theory
of $\CM$-modules defined over $\wt R$. Since $\wt R$ is a valuation ring, this lattice is easily calculated -
see \cite[Ch. 11]{Punb} for a more general setting.
\end{proof}

On the next level of $\CB$-analysis we obtain a finite length lattice.

\begin{prop}\label{l-2}
The following diagram represents the lattice $L_2$ of pp-formulae of the second derivative $T''$ of the theory
$T_{\CM}$ of Cohen--Macaulay modules over $R$.
\end{prop}

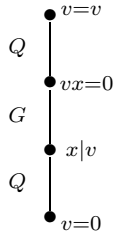
$$
\vcenter{%
\xymatrix@C=20pt@R=20pt{%
*+={\bullet}\ar@{}+<12pt,2pt>*{_{v=v}}\ar@{-}[d]\ar@{}+<-12pt,-12pt>*{_{Q}}\\
*+={\bullet}\ar@{}+<14pt,0pt>*{_{vx=0}}\ar@{-}[d]\ar@{}+<-12pt,-12pt>*{_{G}}\\
*+={\bullet}\ar@{}+<12pt,0pt>*{_{x\mid v}}\ar@{-}[d]\ar@{}+<-12pt,-12pt>*{_{Q}}\\
*+={\bullet}\ar@{}+<12pt,-2pt>*{_{v=0}}
}}
$$

\vspace{3mm}

On this diagram we match simple intervals with points of the Ziegler spectrum they isolate.

\begin{proof}
From the previous section it follows that $T''$ is the closure of points $Q$ and $G$ of $\CB$-rank $2$.
Because $y$ acts as an automorphism on both modules, we conclude that $y^{-1}$ is a definable scalar of this theory.
Therefore $T''$ is just the theory of $S= F((y))[[x]]/(x^2)$-modules. The last ring is uniserial of length
$2$, hence its lattice of pp-formulas is well known.
\end{proof}

\section{Ringel's quilt}

This combinatorial object was introduced by Ringel \cite{Rin00} (see also \cite{Rin11}), under the name of the
Auslander--Reiten quilt, when investigating the module category of domestic string (finite dimensional) algebras.
It serves as a tool for sewing components of the $\AR$-quiver for better understanding of morphism between modules
from different $\AR$-components. The devices used in implementing this construction are infinitely generated
indecomposable pure injective modules and irreducible morphisms between them.

We will proceed in the same vein. First we need some supply of irreducible maps and almost split sequences.

\begin{lemma}\label{Q-map}
The map $Q\xr{x} G$ is irreducible in the category of $\CM$-modules.
\end{lemma}
\begin{proof}
Suppose that $f: Q\to N$ and $g:N\to G$, where $N$ is a $\CM$-module, are morphisms whose composition is the above
map. If $f$ is a monomorphism it splits, because $Q$ is injective. Otherwise $f$ increases the pp-type of $x\in Q$,
hence (see Figure \ref{fig3}) annihilates this element. Because $N$ is $\CM$, it follows that $f(Qx)= 0$, hence
$\ker f= Qx$, otherwise $g$ fail to exists.

Since $Q/Qx\cong G$ via $1\mapsto x$, we conclude that $f$ induces the map $\bar f: G\to N$ such that the composition
$g\bar f: G\to G$ sends $x$ to $x$, hence is an isomorphism. It follows that $g$ splits.
\end{proof}

Note that this irreducible epimorphism is included into a short exact sequence $0\to G\to Q\xr{x} G\to 0$.
However this sequence is far from being almost split: for instance the natural inclusion $G\subset \wt R$
cannot be factored through $Q$.

Recall that a module $M$ over a $\CM$-ring $S$ is said to be \emph{free on the punctured spectrum}, if its localization
$M_P$ is free over the localized ring $S_P$ for each non-maximal prime ideal $P$. Further $S$ is an \emph{isolated
singularity} if $S_P$ is a regular ring for each non-maximal prime ideal $P$.

By Auslander's result (see \cite[Thm. 13.8]{L-W}) if $M$ is a f.g. $\CM$-module free on the punctured spectrum, then
$M$ is the sink of an $\AR$-sequence in the category of finitely generated $\CM$-modules. For $R$ this applies to
finite modules $I_n$, and it is easily seen that the corresponding $\AR$-sequences retain their defining
properties in the category of all $\CM$-modules.

The module $I_{\fty}$ is not free on the punctured spectrum: if $P= xR$, then $I_{\fty}$ localized at $P$ is the
generic module $G$ which is not free over the ring $R_P= F((y))[[x]]/x^2$. It follows from \cite[Prop. 13.3]{L-W}
that $I_{\fty}$ cannot be a sink of an $\AR$-sequence in the category of finitely generated $\CM$-modules.

However, it follows from another result of Auslander \cite[Thm. 5]{Aus86} (see also \cite{Her92}) that $I_{\fty}$
is the sink of an $\AR$-sequence in the category of all $R$-modules whose source is indecomposable and pure 
injective. Using an approximation of this module in the category of all $\CM$-modules we found the following 
$\AR$-sequence, now in the category of $\CM$-modules. We suppress calculations, because the result is easily verified.

\begin{prop}\label{ar}
The following is an $\AR$-sequence in the category of Cohen--Macaulay modules over $R$.

$$
0\to \wt R\xr{f=\, \bsm y\\ x\esm} \wt R\oplus I_{\fty} \xr{(x, -y)} I_{\fty}\to 0\,.
$$
\end{prop}
\begin{proof}
The exactness of this sequence is easily shown by inspection. Since the endomorphism ring of $I_{\fty}$ is local,
by \cite[Prop. 4.4]{Aus78} it suffices to check that $f$ is left almost split.

Suppose that $h: \wt R\to N$ is a non-split monomorphism, where $N$ is a $\CM$-module. In particular for
$n= h(x)$ we obtain $nx= 0$.

$$
\xymatrix@C=20pt@R=20pt{%
\wt R\ar[r]^(.4)f\ar[d]_h& \wt R\oplus I_{\fty}\ar@{-->}[ld]^u\\
N&
}
$$

\vspace{3mm}

Observe that $f(x)= (xy,0)$ whose pp-type $p$ in $\wt R$ (see Figure \ref{fig3}) is generated by the formulas
$xy^n\in I_{n-1}$, $n= 1, \dots$.

Since $\wt R$ is pure injective and indecomposable by \cite[Prop. 4.3.45]{Preb} it follows that $h$ increases the
pp-type of $x\in \wt R$, hence (see Figure \ref{fig3} again) the pp-type $q$ of $h(x)$ in $N$ contains the formula
$xy\in R$. But then $q$ contains any conjunction $(xy\in R)\wg (xy\in I_n)= (xy^n\in I_{n-1})$, $n= 2, \dots$,
therefore $p\seq q$.

Since $N$ is pure injective, there exists a morphism $u: \wt R\to N$ sending $xy$ to $h(x)$, therefore
$h- fu: \wt R\to N$ satisfies $(h- fu)(x)= 0$. Replacing $h$ by $h- fu$ we may assume that $h(x)= 0$.
Since $N$ is $\CM$, it follows that $h(Qx)= 0$, hence $h$ factors through $\wt R\xr{x} I_{\fty}$, as desired.
\end{proof}

The following corollary is straightforward.

\begin{cor}\label{irr-map}
The following morphisms are irreducible in the category of all Cohen--Macaulay $R$-modules.

1) $\wt R\xr{y} \wt R$; \quad 2) $\wt R\xr{x} I_{\fty}$; \quad and \quad 3) $I_{\fty}\xr{y} I_{\fty}$.
\end{cor}
\begin{proof}
Either by projecting the above $\AR$-sequence, or directly.
\end{proof}

Now we are in a position to draw - see Figure \ref{fig5} - Ringel's quilt of the category of $\CM$-modules, and
the procedure is very similar to what Ringel did in \cite{Rin00}.

\begin{figure}
$$
\vcenter{%
\xymatrix@C=16pt@R=16pt{%
&&&*+[o][F]{\bullet}\ar[dr]_y\ar@{}+<0pt,10pt>*{_{I_{\fty}}}
\ar@{}+<-4pt,4pt>*{.}\ar@{}+<-8pt,8pt>*{.}\ar@{}+<-12pt,12pt>*{.}&&&\\
&&*+[o][F]{\bullet}\ar[dl]\ar[dr]_y\ar@{}+<0pt,10pt>*{_{I_2}}
\ar@{}+<-4pt,4pt>*{.}\ar@{}+<-8pt,8pt>*{.}\ar@{}+<-12pt,12pt>*{.}
\ar@{}+<4pt,4pt>*{.}\ar@{}+<8pt,8pt>*{.}\ar@{}+<12pt,12pt>*{.}&&
*+={\bullet}[dr]\ar@{}+<0pt,10pt>*{_{I_{\fty}}}\ar[dr]_y&&\\
&*+[o][F]{\bullet}\ar[dl]\ar[dr]_y\ar@{}+<0pt,10pt>*{_{I_1}}
\ar@{}+<-4pt,4pt>*{.}\ar@{}+<-8pt,8pt>*{.}\ar@{}+<-12pt,12pt>*{.}&&
*+={\bullet}\ar[dl]\ar[dr]_y\ar@{}+<0pt,10pt>*{_{I_3}}
\ar@{}+<4pt,4pt>*{.}\ar@{}+<8pt,8pt>*{.}\ar@{}+<12pt,12pt>*{.}&&
*+={\bullet}\ar@{}+<0pt,10pt>*{_{I_{\fty}}}
\ar@{}+<4pt,-4pt>*{.}\ar@{}+<8pt,-8pt>*{.}\ar@{}+<12pt,-12pt>*{.}&\\
*+[o][F]{\bullet}\ar[dr]_y\ar@{}+<-10pt,0pt>*{_{R}}\ar@{}+<0pt,10pt>*{.}\ar@{}+<0pt,15pt>*{.}\ar@{}+<0pt,20pt>*{.}
&&*+={\bullet}\ar[dl]\ar[dr]_y\ar@{}+<0pt,10pt>*{_{I_2}}&&
*+={\bullet}\ar@{-}[dl]\ar@{}+<0pt,10pt>*{_{I_4}}
\ar@{}+<4pt,4pt>*{.}\ar@{}+<8pt,8pt>*{.}\ar@{}+<12pt,12pt>*{.}
\ar@{}+<4pt,-4pt>*{.}\ar@{}+<8pt,-8pt>*{.}\ar@{}+<12pt,-12pt>*{.}
&&*+={\bullet}\ar@{}+<0pt,10pt>*{_{G}}&&&&\\
&*+[o][F]{\bullet}\ar[dl]\ar[dr]_y\ar@{}+<0pt,10pt>*{_{I_1}}&&
*+={\bullet}\ar[dl]\ar@{}+<0pt,10pt>*{_{I_3}}\ar@{}+<4pt,-4pt>*{.}\ar@{}+<8pt,-8pt>*{.}\ar@{}+<12pt,-12pt>*{.}&&
*+={\bullet}\ar@{}+<0pt,10pt>*{_{\wt R}}\ar[ld]_(.4)y\ar@{}+<4pt,4pt>*{.}\ar@{}+<8pt,8pt>*{.}\ar@{}+<12pt,12pt>*{.}&\\
*+={\bullet}\ar@{}+<-10pt,0pt>*{_R}\ar[rd]_y&&*+[o][F]{\bullet}\ar@{}+<0pt,10pt>*{_{I_2}}\ar[ld]
\ar@{}+<4pt,-4pt>*{.}\ar@{}+<8pt,-8pt>*{.}\ar@{}+<12pt,-12pt>*{.}
&&*+={\bullet}\ar@{}+<0pt,10pt>*{_{\wt R}}\ar[dl]_(.4)y&&\\
&*+={\bullet}\ar@{}+<0pt,10pt>*{_{I_1}}\ar[dl]\ar@{}+<4pt,-4pt>*{.}\ar@{}+<8pt,-8pt>*{.}\ar@{}+<12pt,-12pt>*{.}
&&*+[o][F]{\bullet}\ar@{}+<0pt,10pt>*{_{\wt R}}\ar[dl]_(.4)y
\ar@{-->} '[r] '[rrrruuu]_x '[ruuuuuu] [uuuuuu]&&&\\
*+={\bullet}\ar@{}+<-10pt,0pt>*{_R}\ar@{}+<0pt,-6pt>*{.}\ar@{}+<0pt,-12pt>*{.}\ar@{}+<0pt,-18pt>*{.}&&
*+={\bullet}\ar@{}+<0pt,10pt>*{_{\wt R}}\ar@{}+<-6pt,-6pt>*{.}\ar@{}+<-12pt,-12pt>*{.}\ar@{}+<-18pt,-18pt>*{.}
\ar@{-->} '[rr] '[rrrrrruuuu]_x '[rrrruuuuuu] [rruuuuuu]&&&&&&&&\\
&&&&&&&&&\\
*+={\bullet}\ar@{}+<0pt,10pt>*{_Q}&&&&&\\
&&&&&&
}}
$$
\caption{}\label{fig5}
\end{figure}

\vspace{3mm}

Namely first we change the metric to include the $\AR$-quiver of $R$ (see Figure \ref{fig3}) into a finite region
- it will be contained in the ambient triangle erected vertically on the vertex $Q$. Now compactify this
triangle by adjusting the copies of indecomposable pure injective $\CM$-modules to its boarder. Namely add copies
of $\wt R$ as direct limits along rays of irreducible maps. We also add the irreducible maps $\wt R\xr{y} \wt R$,
which are the direct limits of inclusion maps $I_{n+1}\subset I_n$ in the $\AR$-quiver; and adjust $Q$ as the direct
limit of this ray between the $\wt R$ (going in the southwestern direction). Further add $G$ as a limit of the ray
of irreducible maps $I_{\fty}\xr{y} I_{\fty}$ and adjust remaining irreducible morphisms $\wt R\xr{x} I_{\fty}$ and
$Q\xr{x} G$.

We use irreducible morphisms to connect the points on the boundary, hence give an orientation to sides of the triangle
to glue them.

$$
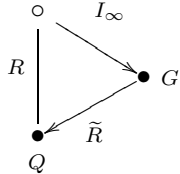
\xymatrix@C=10pt@R=14pt{%
*+{\circ}\ar[drr]\ar@{-}[dd]+<0pt,5pt>&\ar@{}+<6pt,0pt>*{_{I_\fty}}&\\
\ar@{}+<-8pt,4pt>*{_R}&&*+={\bullet}\ar[lld]\ar@{}+<10pt,0pt>*{_G}\\
*+={\bullet}\ar@{}+<0pt,-10pt>*{_Q}& \ar@{}+<0pt,2pt>*{_{\wt R}}&
}
$$

\vspace{3mm}

We see that topologically Ringel's quilt of $R$ lives on the M\"obius stripe. A small deficiency is the empty point
on the boundary of the triangle which can be easily amended: adjust the zero module to this point.

From Ringel's quilt (see Figure \ref{fig5}) we get a better understanding of morphisms between finitely generated
$\CM$-modules. For instance, starting from $R$ one can move along the southeastern line, then live the
$\AR$-component in $\wt R$ and enter the other components through $I_{\fty}$. After that one could reenter the
first component by embedding $I_{\fty}$ in $R$, - the morphism that heads in the southwestern direction.
This way we have completed the revolution.

Observe that $x\in \wt R$ is divisible by any power of $y$. Thus every morphism from $\wt R$ to a finitely
generated $\CM$-module annihilates the submodule $Qx$ of $\wt R$, therefore factors through $I_{\fty}$. Thus
$I_{\fty}$ is a point where $\wt R$ enters the category of finitely generated $\CM$-modules.

There is one misleading point concerning Figure \ref{fig5}. The choice where the morphism $\wt R\to I_{\fty}$
enters the second component is completely arbitrary, but the construction becomes rigid as soon as such choice
has been made. Thus saying 'a revolution' is a bit ambiguous, the corresponding walk is better represented
by the following diagram.

$$
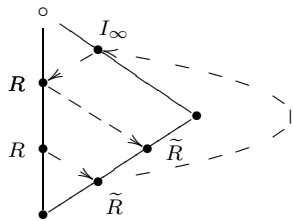
\def\objectstyle{\scriptstyle}
\xymatrix@C=10pt@R=4pt{%
*+{\circ}\ar@{-}[dddrrr]\ar@{-}[dddddd]&&\\
&*+={\bullet}\ar@{-->}[dl]\ar@{}+<6pt,8pt>*{I_{\fty}}&\\
*+={\bullet}\ar@{}+<-10pt,0pt>*{R}\ar@{-->}[ddrr]\ar@{}+<-10pt,0pt>*{R}&&&\\
&&&*+={\bullet}\ar@{-}[dddlll]\\
*+={\bullet}\ar@{-->}[dr]\ar@{}+<-10pt,0pt>*{R}&&*+={\bullet}\ar@{}+<10pt,0pt>*{\wt R}&\\
&*+={\bullet}\ar@{}+<6pt,-8pt>*{\wt R}\ar@{-->}@/_6pc/[uuuu]\\
*+={\bullet}&&&\\
&&&&
}
$$

\vspace{3mm}

It may be easier to grasp such walks in the following representations of Ringel's quilt - see Fugure \ref{fig6},
where we dashed the fundamental domain. Thus it is obtained as a quotient of the vertical strip on which a cyclic group
acts by a glide-reflection. Again by adding the zero module we will get rid of small triangular deficiencies on
this diagram, but one should add maps $G\to 0$ and $0\to I_{\fty}$ which are not irreducible.

\begin{figure}
$$
\vcenter{%
\xymatrix@C=20pt@R=18pt{%
*+={\bullet}\ar[dr]_y\ar@{}+<0pt,-10pt>*{_R}\ar@{--}
\ar@{--} '[]+<0pt,18pt>'[rrrrrddddd]+<18pt,0pt> '[rddddddddd]+<0pt,-14pt>'[dddddddd]+<-10pt,-4pt>  '[]+<-10pt,8pt>
[]+<0pt,18pt>
&&&&&&&&\\
&*+={\bullet}\ar[dl]\ar[dr]_y\ar@{}+<0pt,10pt>*{_{I_1}}&&&&&&\\
*+[o][F]{\bullet}\ar[dr]_y\ar@{}+<0pt,10pt>*{_R}&&*+={\bullet}\ar[dl]\ar[dr]_y\ar@{}+<0pt,10pt>*{_{I_2}}&&&&&\\
&*+[o][F]{\bullet}\ar[dl]\ar[dr]_y\ar@{}+<0pt,10pt>*{_{I_1}}&&
*+={\bullet}\ar[dl]\ar@{}+<0pt,10pt>*{_{I_3}}\ar@{}+<4pt,-4pt>*{.}\ar@{}+<8pt,-8pt>*{.}\ar@{}+<12pt,-12pt>*{.}&&&&\\
*+={\bullet}\ar[dr]_y\ar@{}+<0pt,10pt>*{_R}&&
*+[o][F]{\bullet}\ar[dl]\ar@{}+<0pt,10pt>*{_{I_2}}\ar@{}+<4pt,-4pt>*{.}\ar@{}+<8pt,-8pt>*{.}\ar@{}+<12pt,-12pt>*{.}&&
*+={\bullet}\ar[dl]_y\ar[dr]_x\ar@{}+<0pt,10pt>*{_{\wt R}}&&&\\
&*+={\bullet}\ar[dl]\ar@{}+<0pt,10pt>*{_{I_1}}\ar@{}+<4pt,-4pt>*{.}\ar@{}+<8pt,-8pt>*{.}\ar@{}+<12pt,-12pt>*{.}&&
*+[o][F]{\bullet}\ar[dl]_y\ar[dr]_x\ar@{}+<0pt,10pt>*{_{\wt R}}&&
*+={\bullet}\ar[dl]_y\ar@{}+<4pt,6pt>*{_{I_{\fty}}}
\ar@{}+<4pt,-4pt>*{.}\ar@{}+<8pt,-8pt>*{.}\ar@{}+<12pt,-12pt>*{.}&&\\
*+={\bullet}\ar@{}+<0pt,10pt>*{_R}\ar@{}+<4pt,-4pt>*{.}\ar@{}+<8pt,-8pt>*{.}\ar@{}+<12pt,-12pt>*{.}&&
*+={\bullet}\ar[dl]_y\ar[dr]_x\ar@{}+<0pt,10pt>*{_{\wt R}}&&
*+[o][F]{\bullet}\ar[dl]_y\ar@{}+<0pt,10pt>*{_{I_{\fty}}}
\ar@{}+<4pt,-4pt>*{.}\ar@{}+<8pt,-8pt>*{.}\ar@{}+<12pt,-12pt>*{.}&&*+={\bullet}\ar[dl]_y\ar@{}+<0pt,-10pt>*{_R}
\ar@{--} '[]+<10pt,4pt>'[dddddddd]+<10pt,0pt> '[lddddddddd]+<-4pt,-14pt>'[lllllddddd]+<-22pt,0pt>  '[]+<0pt,16pt>
[]+<10pt,4pt>&&&&&&&&
\\
&*+={\bullet}\ar[dr]_x\ar@{}+<0pt,10pt>*{_{\wt R}}\ar@{}+<-4pt,-4pt>*{.}\ar@{}+<-8pt,-8pt>*{.}\ar@{}+<-12pt,-12pt>*{.}&&
*+={\bullet}\ar[dl]_y\ar@{}+<0pt,10pt>*{_{I_{\fty}}}\ar@{}+<4pt,-4pt>*{.}\ar@{}+<8pt,-8pt>*{.}\ar@{}+<12pt,-12pt>*{.}&&
*+[o][F]{\bullet}\ar[dl]_y\ar[dr]\ar@{}+<0pt,10pt>*{_{I_1}}&&\\
*+={\bullet}\ar[dr]_x\ar@{}+<0pt,10pt>*{_Q}&&
*+={\bullet}\ar@{}+<0pt,12pt>*{_{I_{\fty}}}\ar@{}+<4pt,-4pt>*{.}\ar@{}+<8pt,-8pt>*{.}\ar@{}+<12pt,-12pt>*{.}
\ar@{}+<-4pt,-4pt>*{.}\ar@{}+<-8pt,-8pt>*{.}\ar@{}+<-12pt,-12pt>*{.}&&
*+={\bullet}\ar[dl]_y\ar[dr]\ar@{}+<0pt,10pt>*{_{I_2}}&&*+[o][F]{\bullet}\ar[dl]_y\ar@{}+<0pt,10pt>*{_R}&\\
&*+={\bullet}\ar@{}+<0pt,10pt>*{_G}\ar@{}+<4pt,-4pt>*{.}\ar@{}+<8pt,-8pt>*{.}\ar@{}+<12pt,-12pt>*{.}&&
*+={\bullet}\ar[dr]\ar@{}+<0pt,10pt>*{_{I_3}}\ar@{}+<-4pt,-4pt>*{.}\ar@{}+<-8pt,-8pt>*{.}\ar@{}+<-12pt,-12pt>*{.}&&
*+={\bullet}\ar[dl]_y\ar[dr]\ar@{}+<0pt,10pt>*{_{I_1}}&&\\
&&*+={\bullet}\ar[dl]_x\ar[dr]^y\ar@{}+<0pt,12pt>*{_{\wt R}}&&
*+={\bullet}\ar[dr]\ar@{}+<0pt,10pt>*{_{I_2}}\ar@{}+<-4pt,-4pt>*{.}\ar@{}+<-8pt,-8pt>*{.}\ar@{}+<-12pt,-12pt>*{.}&&
*+={\bullet}\ar[dl]_y\ar@{}+<0pt,10pt>*{_R}\\
&*+={\bullet}\ar[dr]^y\ar@{}+<0pt,12pt>*{_{I_{\fty}}}\ar@{}+<-4pt,-4pt>*{.}\ar@{}+<-8pt,-8pt>*{.}\ar@{}+<-12pt,-12pt>*{.}&&
*+={\bullet}\ar[dr]^y\ar[dl]_x\ar@{}+<0pt,12pt>*{_{\wt R}}&&
*+={\bullet}\ar@{}+<0pt,10pt>*{_{I_1}}\ar[dr]\ar@{}+<-4pt,-4pt>*{.}\ar@{}+<-8pt,-8pt>*{.}\ar@{}+<-12pt,-12pt>*{.}&\\
*+={\bullet}\ar[dr]\ar@{}+<0pt,-10pt>*{_R}
\ar@{--} '[]+<0pt,18pt>'[rrrrrddddd]+<18pt,0pt> '[rddddddddd]+<0pt,-14pt>'[dddddddd]+<-10pt,-4pt>  '[]+<-10pt,8pt>
[]+<0pt,18pt>
&&
*+={\bullet}\ar[dr]^y\ar@{}+<0pt,12pt>*{_{I_{\fty}}}\ar@{}+<-4pt,-4pt>*{.}\ar@{}+<-8pt,-8pt>*{.}\ar@{}+<-12pt,-12pt>*{.}&&
*+={\bullet}\ar[dl]_x\ar@{}+<0pt,12pt>*{_{\wt R}}\ar[dr]&&*+={\bullet}\ar@{}+<0pt,12pt>*{_R}
\ar@{}+<-4pt,-4pt>*{.}\ar@{}+<-8pt,-8pt>*{.}\ar@{}+<-12pt,-12pt>*{.}&&\\
&*+={\bullet}\ar[dr]\ar[dl]\ar@{}+<0pt,10pt>*{_{I_1}}&&
*+={\bullet}\ar@{}+<0pt,12pt>*{_{I_{\fty}}}\ar[dr]\ar@{}+<-4pt,-4pt>*{.}\ar@{}+<-8pt,-8pt>*{.}\ar@{}+<-12pt,-12pt>*{.}&&
*+={\bullet}\ar[dl]_x\ar@{}+<0pt,12pt>*{_{\wt R}}\ar@{}+<4pt,-4pt>*{.}\ar@{}+<8pt,-8pt>*{.}\ar@{}+<12pt,-12pt>*{.}&&&&\\
*+={\bullet}\ar[dr]\ar@{}+<0pt,10pt>*{_R}&&*+={\bullet}\ar[dr]\ar[dl]\ar@{}+<0pt,10pt>*{_{I_2}}&&
*+={\bullet}\ar@{}+<0pt,12pt>*{_{I_{\fty}}}\ar@{}+<4pt,-4pt>*{.}\ar@{}+<8pt,-8pt>*{.}\ar@{}+<12pt,-12pt>*{.}
\ar@{}+<-4pt,-4pt>*{.}\ar@{}+<-8pt,-8pt>*{.}\ar@{}+<-12pt,-12pt>*{.}&&
*+={\bullet}\ar@{}+<0pt,12pt>*{_Q}\ar[dl]_x&&\\
&*+={\bullet}\ar[dl]\ar[dr]\ar@{}+<0pt,10pt>*{_{I_1}}&&
*+={\bullet}\ar[dl]\ar@{}+<0pt,10pt>*{_{I_3}}\ar@{}+<4pt,-4pt>*{.}\ar@{}+<8pt,-8pt>*{.}\ar@{}+<12pt,-12pt>*{.}&&
*+={\bullet}\ar@{}+<0pt,12pt>*{_G}&&&\\
*+={\bullet}\ar[dr]\ar@{}+<0pt,10pt>*{_R}&&
*+={\bullet}\ar[dl]\ar@{}+<0pt,10pt>*{_{I_2}}\ar@{}+<4pt,-4pt>*{.}\ar@{}+<8pt,-8pt>*{.}\ar@{}+<12pt,-12pt>*{.}&&
*+={\bullet}\ar[dr]^x\ar[dl]_y\ar@{}+<0pt,10pt>*{_{\wt R}}&&&&&\\
&*+={\bullet}\ar[dl]\ar@{}+<0pt,10pt>*{_{I_1}}\ar@{}+<4pt,-4pt>*{.}\ar@{}+<8pt,-8pt>*{.}\ar@{}+<12pt,-12pt>*{.}&&
*+={\bullet}\ar[dr]^x\ar[dl]_y\ar@{}+<0pt,10pt>*{_{\wt R}}&&
*+={\bullet}\ar[dl]_y\ar@{}+<0pt,10pt>*{_{I_{\fty}}}&&&&\\
*+={\bullet}\ar@{}+<0pt,10pt>*{_R}\ar@{}+<4pt,-4pt>*{.}\ar@{}+<8pt,-8pt>*{.}\ar@{}+<12pt,-12pt>*{.}&&
*+={\bullet}\ar[dr]^x\ar@{}+<0pt,10pt>*{_{\wt R}}\ar[dl]_y&&
*+={\bullet}\ar[dl]_y\ar@{}+<0pt,10pt>*{_{I_{\fty}}}&&&&&\\
&*+={\bullet}\ar[dr]^x\ar@{}+<0pt,10pt>*{_{\wt R}}\ar@{}+<-4pt,-4pt>*{.}\ar@{}+<-8pt,-8pt>*{.}\ar@{}+<-12pt,-12pt>*{.}&&
*+={\bullet}\ar[dl]_y\ar@{}+<0pt,10pt>*{_{I_{\fty}}}&&&&&&\\
*+={\bullet}\ar@{}+<0pt,10pt>*{_Q}\ar[rd]^x&&
*+={\bullet}\ar@{}+<0pt,10pt>*{_{I_{\fty}}}\ar@{}+<-4pt,-4pt>*{.}\ar@{}+<-8pt,-8pt>*{.}\ar@{}+<-12pt,-12pt>*{.}&&&&&&\\
&*+={\bullet}\ar@{}+<0pt,10pt>*{_G}&&&&&&
}}
$$
\caption{}\label{fig6}
\end{figure}

\vspace{5mm}

\section{The infinite radical}

In this section we will calculate the nilpotency index of the radical of the category of finitely generated
$\CM$-modules over $R$.

Define the radical, written $\rad$,  of this category as a set of all morphisms $f: M\to N$ between finitely
generated $\CM$-modules such that for every indecomposable finitely generated $\CM$-module $K$, any composition
$K\to M\xr{f} N\to K$ is not invertible in the local ring $\End(K)$. Clearly $\rad$ is a 2-sided ideal in this
category.

Decomposing $M$ and $N$ as direct sums of indecomposable modules we represent $f$ as a finite matrix $f_{ij}$
whose entries are morphisms between indecomposable modules. Clearly $f\in \rad$ iff each $f_{ij}\in \rad$.
Thus when describing the radical is usually suffices to consider morphisms between indecomposable modules.

Furthermore, because each indecomposable finitely generated $\CM$-module is pure injective, from
\cite[Prop. 4.3.45]{Preb} it follows that a map $f: M\to N$ between such indecomposables belong to the radical
iff it increases the pp-type of one (equivalently any) nonzero element of $M$.

Next we will show that there are enough irreducible morphisms in our category.

\begin{lemma}\label{gen}
$\rad$ is generated by irreducible morphisms.
\end{lemma}
\begin{proof}
It suffices to prove that every morphism $f: M\to N$ between indecomposable finitely generated $\CM$-modules
which lies in the radical belongs to the ideal generated by irreducibles. If $M$ is a finite point $I_n$, $n\geq 1$,
then this follows from the fact that $I_n$ is the source of a left almost split map. Suppose that $M= I_0= R$ and look
at at $f(x)$. From Figure \ref{fig3} we conclude that $f$ factors through the irreducible morphism $R\xr{y} I_1$.

It remains to consider the case $M= I_{\fty}$. Observing the same figure we conclude that either $f$ factors through
the irreducible morphism $I_{\fty}\xr{y} I_{\fty}$, or $f$ has $I_n$ as a target. In the latter case $f$ factors
through the irreducible inclusion $I_{n+1}\subset I_n$.
\end{proof}

We define the \emph{infinite powers of the radical}, $\rad^{\lam}$, by a transfinite induction starting from
$\rad^1= \rad$. If $\lam$ is a limit ordinal, then set $\rad^{\lam}= \cap_{\mu< \lam} \rad^{\mu}$; for instance
\emph{the infinite radical} $\rad^{\om}$ equals $\cap_n \rad^n$. Finally if $\lam= \mu+ m$ is a successor ordinal,
then define $\rad^{\lam}= (\rad^{\mu})^{m+1}$, for instance $\rad^{\om+1}= (\rad^{\om})^2$. The least $\lam$ such
that $\rad^{\lam}= 0$ (if exists) is called \emph{the nilpotency index} of the radical. For instance, the nilpotency
index equals $\om+ 1$ if there exists a nonzero morphism in the infinite radical, but for any two such morphisms
their composition equals zero.

There is a heuristic (see \cite{Sch00b}) intuition how to use the $m$-dimension of the lattice $L$ of pp-formulae
to guess the nilpotency index of the radical. Namely suppose that the $m$-dimension of $L$ equals $\be+ 1$,
hence $L_{\be}$ is the last nontrivial lattice in the $m$-dimension analysis. If the length of $L_{\be}$ equals
$n$, then one could expect $\om \be+ n-1$ as the value of the nilpotency index. Since in our case the
$m$-dimension equals $2$ and $n= 3$, we should get $\om+ 2$, which is confirmed below.

The following remark is straightforward.

\begin{remark}\label{inf-rad}
Suppose that $f: M\to N$ is a morphism of indecomposable finitely generated $\CM$-modules. If $f\in \rad^{\om}$
and $mx=0$ for $m\in M$ then $f(m)= 0$.
\end{remark}
\begin{proof}
Because $mx= 0$ the pp-type of $m$ in $M$ is visible on Figure \ref{fig2}. Since $f$ belongs to the infinite
radical, for each $k$, it can be written as a sum of products $f_1\cdot\ldots\cdot f_k$ of radical maps between
indecomposable finitely generated $\CM$-modules. Since each $f_i$ lower the pp-type of an element on the figure,
we derive the desired.
\end{proof}

In fact the above condition characterizes maps in the infinite radical, but we prefer to be even more precise
when describing such morphism between indecomposable modules.

\begin{lemma}\label{infrad-desc}
1) Each morphism $I_n\to I_{\fty}$ and $I_{\fty}\to I_n$ belongs to $\rad^{\om}$.

2) A morphism $h: I_m\to I_n$ belongs to $\rad^{\om}$ iff it factors through $I_{\fty}$.

3) Each nonzero endomorphism of $I_{\fty}$ does not belong to $\rad^{\om}$.
\end{lemma}
\begin{proof}
1) Let $f: I_n\to I_{\fty}$. If $I_n\xr{\bsm 1\\y\esm} I_{n-1}\oplus I_{n+1}$ is the left almost split map,
then $f$ factors as $f_1+ y f_2$ where $f_1: I_{n-1}\to I_{\fty}$ and $f_2: I_{n+1}\to I_{\fty}$, hence we could
proceed by induction.

Furthermore each map $g: I_{\fty}\to I_m$ is given (up to a multiplicative unit) by multiplication by $y^k$,
therefore it factors as $I_{\fty}\xr{y^k} I_{m+1}\subset I_m$, and proceed by induction.

2) Suppose that $h$ belongs to the infinite radical. Recall that $I_m$ is generated by $x$ and $y^m$.
From Remark \ref{inf-rad} it follows that $f(x R)=0$. Since $I_m/x R\cong I_{\fty}$ we obtain the desired.

The converse is easy, say, because $f$ factors through the infinitely generated module.

3) We may assume that a nonzero endomorphism $f$ of $I_{\fty}$ is given by multiplication by $y^k$. For
$m= x\in I_{\fty}$ we have $mx=0$ but $f(m)\neq 0$, hence $f\notin \rad^{\om}$ by Corollary \ref{inf-rad}.
\end{proof}

Now we are in a position to prove the main result of this section.

\begin{theorem}\label{om2}
The nilpotency index of the radical of the category of finitely generated Cohen--Macaulay $R$-modules equals
$\om+ 2$.
\end{theorem}
\begin{proof}
One half of this theorem is easy. Namely let $f: R\xr{x} I_{\fty}$ and let $g$ be the inclusion $I_{\fty}\subset R$.
Then both maps belong to $\rad^{\om}$ by Lemma \ref{infrad-desc}. From $gf\neq 0$ we conclude that
$\rad^{\om+ 1}= (\rad^{\om})^2\neq 0$.

To prove the converse it suffices to show that any composition $M\xr{f} N\xr{g} K\xr{h} D$ of maps in $\rad^{\om}$
between indecomposable finite generated $\CM$-modules is zero. By Lemma \ref{infrad-desc} this is the case when
$I_{\fty}$ occurs among these modules at least twice.

Suppose that $I_{\fty}$ occurs exactly once. If $M= I_{\fty}$ then $N= I_m$, $K= I_n$, hence the morphism
$N\to K$ factors through $I_{\fty}$.

$$
\vcenter{%
\xymatrix@C=10pt@R=14pt{%
*+={I_{\fty}}\ar[rr]^f&&*+={N}\ar[rr]^g\ar@{-->}[dr]_u&&*+{K}\\
&&&*+={I_{\fty}}\ar@{-->}[ur]_v&
}}
$$

\vspace{2mm}

Because $f\in \rad^{\om}$, by Lemma \ref{inf-rad} we conclude that $uf= 0$, hence the composite map is zero.
Furthermore the proof is similar when $I_{\fty}$ is positioned as $N, K$ or $D$.

Finally if $I_{\fty}$ does not occur in the above sequence then we conclude from the following diagram.

$$
\vcenter{%
\xymatrix@C=10pt@R=14pt{%
*+={M}\ar[rr]^f\ar@{-->}[dr]&&*+={N}\ar[rr]^g\ar@{-->}[dr]&&*+={K}\ar[rr]^h\ar@{-->}[dr]&&*+={D}\\
&*+={I_{\fty}}\ar@{-->}[ur]&&*+={I_{\fty}}\ar@{-->}[ur]&&*+={I_{\fty}}\ar@{-->}[ur]&&
}}
$$

\vspace{2mm}

Namely the composite morphism from the first to the last copy of $I_{\fty}$ factors through $g$, hence belongs
to $\rad^{\om}$, therefore equals zero by Lemma \ref{infrad-desc}.
\end{proof}

\section{Heigher dimensions}

Recall that the $A_{\fty}$ singularity of dimension $d\geq 1$ is defined (see \cite[p. 253]{L-W}) as the hypersurface
$R_d= F[[x_0, \dots, x_d]]/(x_1^2+ \dots + x_d^2)$. A remarkable result, the so-called Kn\"orrer's periodicity
(see \cite[Sect. 8.3]{L-W}), claims that the stable categories of finitely generated $\CM$-modules over these rings
are equivalent, for even, and separately for odd $d$'s.

Suppose that one would like to describe the $\CM$-part of the Ziegler spectrum over these rings. We will show
that even for $d= 2$ this is a problem of enormous complexity. Namely after changing variables we see that
$R_2$ is isomorphic to the ring $F[[x,y,z]/(xz)$. Furthermore the $\CM$-condition means $\Hom(F,M)= 0$
and $\Ext^1(F,M)= 0$.

We will consider a very particular case of this problem. Namely let $T$ denote the closed subset of $\ZCM$
over $R_2$ consisting of modules annihilated by $z$. It is easily seen that those are exactly the
Cohen--Macaulay modules over the 2-dimensional regular ring $S= F[[x,y]]$. Further $S_S$ is the unique
indecomposable finitely generated $\CM$-module over $S$.

The following remark refutes a conjecture made in \cite[Conj. 6.9]{H-P}.

\begin{remark}
Finitely generated points are not dense in the Cohen--Macaulay part of the Ziegler spectrum of $F[[x,y]]$.
\end{remark}
\begin{proof}
If $E$ is the injective envelope of $S/(x+y)S$ then this module is clearly $\CM$. It opens the interval
$[v=0, v(x+y)=0]$, but this interval is closed on $S$, hence $E$ is not in the closure of finitely generated
points.
\end{proof}

Furthermore there is a little hope to classify all points of $\ZCM$ over $S$. Namely let $P= (x^3- y^2)$ be the
prime ideal and let $S'$ be the localization $S_P$. Since $x, y\notin P$ these elements are invertible in $S'$,
hence act as automorphisms on any $S'$-module; in particular every finite generated $S'$-module is an
(infinitely generated) $\CM$-module over $S$. However $S'$ admits as a factor ring the Drozd ring
$F[[x,y]]/(x^2,xy^2,y^3)$, and this 5-dimensional algebra (see \cite[p. 346]{K-L}) is wild.

Thus one should put additional restrictions on non-finitely generated $\CM$-modules over $S$ to classify them.
One natural condition would be the Hochster-like $M\neq M\mm$. However we do not know the answer to the following
question.

\begin{ques}
Does there exist a non-finitely generated indecomposable pure injective Cohen--Macaulay module $M$ over the ring
$F[[x,y]]$ such that $M\neq M\mm$?
\end{ques}

It seems to be more natural to force finitely generated points to be dense, therefore consider the closure, in
the Ziegler spectrum, of finitely generated $\CM$-modules.

\begin{conj}
Let $R_d$, $d\geq 1$ be the $d$-dimensional $A_{\fty}$-singularity over an algebraically closed field $F$ and let
$T_{CM}$ denote the theory of finitely generated Cohen--Macaulay modules over this ring. Then the Cantor--Bendixson
rank of $T_{CM}$ equals $2d$, and the same is the value of the Krull--Gabriel dimension of the definable category
generated by finitely generated Cohen--Macaulay $R_d$-modules.
\end{conj}

The intuition behind this conjecture is the following. When $d$ increases, the stable part of this category remains
'the same', but for the projective part of this category we should add an extra 2 for each addition dimension.


\end{document}